\documentclass[12pt,reqno,dvipsnames]{amsart}
\usepackage{fullpage}
\usepackage{times}
\usepackage{colonequals}
\usepackage{amsmath,amssymb,amsthm,url}
\usepackage[dvipsnames]{xcolor}
\usepackage[utf8]{inputenc}
\usepackage[english]{babel}
\usepackage[alphabetic]{amsrefs}
\usepackage{bbm}
\usepackage{enumerate}
\usepackage{xcolor}
\usepackage{tikz, tikz-cd}
\definecolor{purpleHTML}{HTML}{8000FF}
\tikzset{
  red_vertex/.style={circle, thick, draw=red, fill=white},
  blue_vertex/.style={circle, thick, draw=NavyBlue},
  green_vertex/.style={circle, thick, draw=ForestGreen},
  purple_vertex/.style={circle, thick, draw=purpleHTML},
  vertex/.style={circle, thick, draw=black},
  edge/.style={thick},
  red_edge/.style={thick, draw=red},
  blue_edge/.style={thick, draw=NavyBlue},
  green_edge/.style={thick, draw=ForestGreen},
  purple_edge/.style={thick, draw=purpleHTML},
  purple_dotted_edge/.style={thick, dash pattern=on 3pt off 2pt, draw=purpleHTML}
}
\usepackage{graphicx}
\usepackage{centernot}
\usepackage{mathrsfs}
\usepackage[colorlinks=true, pdfstartview=FitH, linkcolor=blue, citecolor=blue, urlcolor=blue]{hyperref}

\newtheorem{theorem}{Theorem}[section]
\newtheorem{lemma}[theorem]{Lemma}
\newtheorem{prop}[theorem]{Proposition}
\newtheorem{cor}[theorem]{Corollary}

\theoremstyle{definition}
\newtheorem{defin}[theorem]{Definition}

\newcommand{\Int}{\operatorname{int}}
\newcommand{\Ext}{\operatorname{ext}}

\title{Counting Subgraphs of Coloring Graphs using Shadow Graphs}

\author{Simon MacLean}
\address{Department of Mathematics \& Computer Science \\ Santa Clara University \\ CA 95050 \\ USA}
\email{smaclean@scu.edu}

\subjclass[2020]{Primary: 05C15, 05C31; Secondary: 05C05, 05C07, 05C60}
\keywords{coloring graphs, chromatic polynomials, graph invariants, generalized degree sequences}

\begin{document}

\begin{abstract}
Given a graph $G$, the $k$-coloring graph $\mathcal{C}_k(G)$ is constructed by selecting proper $k$-colorings of $G$ as vertices, with an edge between two colorings if they differ in the color of exactly one vertex. The number of vertices in $\mathcal{C}_k(G)$ is the famous chromatic polynomial of $G$. Asgarli, Krehbiel, Levinson and Russell showed that for any subgraph $H$, the number of induced copies of $H$ in $\mathcal{C}_k(G)$ is a polynomial function in $k$. Hogan, Scott, Tamitegama, and Tan found a shorter proof for polynomiality of these chromatic $H$-polynomials. In this paper, we provide a method of constructing these polynomials explicitly in terms of chromatic polynomials of \emph{shadow graphs}. We illustrate the practicality of our formulas by computing an explicit formula for $H$-polynomial for trees when $H=Q_d$ is an arbitrary hypercube, a task which does not seem approachable from previous methods. The coefficients of the resulting polynomials feature \emph{generalized degree sequences} introduced by Crew. In the special case when $H=P_2$, the corresponding polynomial is dubbed the \emph{chromatic pairs polynomial}. We present a pair of graphs $G_1$ and $G_2$ sharing the same chromatic pairs polynomial but different chromatic polynomials, disproving a conjecture raised by Asgarli, Krehbiel, Levinson and Russell. 
\end{abstract}

\maketitle

\section{Introduction}

The chromatic polynomial $\pi_G(k)$ of a simple graph $G$ counts the number of ways to properly color $G$ using $k$ colors. This polynomial invariant was introduced by Birkhoff~\cite{Bir12} in an attempt to prove the Four Color Theorem. This concept is foundational in graph theory and combinatorics. Instead of simply counting the number of colorings, we can also keep track of how two different colorings are related. One such framework, known as a reconfiguration of $k$-colorings, declares a precise adjacency between two colorings. More precisely, the \emph{coloring graph} $\mathcal{C}_k(G)$ represents all proper colorings of $G$, with edges between colorings that differ at exactly one vertex. Viewed through this lens, the chromatic polynomial of $G$ is the number of vertices in $C_k(G)$ as $k$ varies.

Working with $\mathcal{C}_k(G)$ allows us to generalize the chromatic polynomial by counting other subgraphs of $\mathcal{C}_k(G)$. Asgarli, Krehbiel, Levinson, and Russell \cite{AKLR25} introduce the notation $\pi_G^{(H)}(k)$ to denote the function that counts occurrences of a graph $H$ as an induced subgraph of $\mathcal{C}_k(G)$. There are now two different proofs that $\pi_G^{(H)}(k)$ is indeed a polynomial function of $k$: the first proof is in the original paper \cite{AKLR25} and the second proof appeared in a subsequent work by Hogan, Scott, Tamitegama, and Tan \cite{HSTT24}. In this paper, we offer a new proof of the polynomiality which has the advantage of being constructive in nature. We present a method using \emph{shadow graphs} to calculate these polynomials and derive explicit formulas. 

\begin{theorem}
Given a fixed connected graph $H$, there exists an algorithm (whose complexity depends on $H$) which applied to any base graph $G$ produces a sequence of graphs $G_1, G_2, \ldots, G_r$, derived from $G$ such that $\pi_G^{(H)}(k)$ is a rational linear combination of $\pi_{G_1}(k), \pi_{G_2}(k), \ldots, \pi_{G_r}(k)$.     
\end{theorem}

We also focus on the special case when $H=Q_d\cong \mathop{\square}_{i=1}^d K_2$ is the hypercube graph and $G=T$, an arbitrary tree. We discuss how $\pi_{T}^{(Q_d)}(k)$ relates to the generalized degree sequence of trees, as studied by Crew in \cite{Cre20}, and their limitations as complete invariants.

Finally, we present a counterexample to the conjecture of Asgarli, Krehbiel, Levinson and Russell. The authors conjectured \cite{AKLR25}*{Conjecture 5.2} that if two graphs $G_1$ and $G_2$ satisfy $\pi_{G_1}^{(P_2)}(k)=\pi_{G_2}^{(P_2)}(k)$, then $\pi_{G_1}(k)=\pi_{G_2}(k)$. The counterexample below was found quickly using a computer program by applying the formulas involving shadow graphs developed in this paper.  
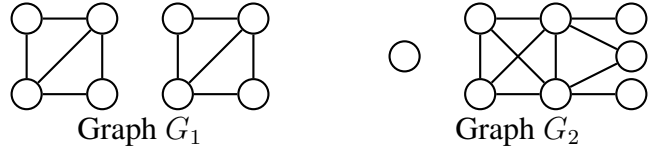
\begin{figure}[ht]
    \centering
    \begin{tikzpicture}
        \node[vertex] (A1) at (0,0) {};
        \node[vertex] (B1) at (1,0) {};
        \node[vertex] (C1) at (1,1) {};
        \node[vertex] (D1) at (0,1) {};
        \draw[edge] (A1) -- (B1) -- (C1) -- (D1) -- (A1) -- (C1);
        \node[vertex] (A2) at (2,0) {};
        \node[vertex] (B2) at (3,0) {};
        \node[vertex] (C2) at (3,1) {};
        \node[vertex] (D2) at (2,1) {};
        \draw[edge] (A2) -- (B2) -- (C2) -- (D2) -- (A2) -- (C2);
        \node[vertex] (A) at (6,0) {};
        \node[vertex] (B) at (7,0) {};
        \node[vertex] (C) at (7,1) {};
        \node[vertex] (D) at (6,1) {};
        \node[vertex] (E) at (8,0.5) {};
        \node[vertex] (F) at (5, 0.5) {};
        \node[vertex] (G) at (8,0) {};
        \node[vertex] (H) at (8,1) {};
        \draw[edge] (A) -- (B) -- (C) -- (D) -- (A) -- (C);
        \draw[edge] (B) -- (D);
        \draw[edge] (B) -- (E);
        \draw[edge] (C) -- (E);
        \draw[edge] (G) -- (B);
        \draw[edge] (H) -- (C);

        \node at (6.5, -0.5) {Graph $ G_2 $};
        \node at (1.5, -0.5) {Graph $ G_1 $};
        
    \end{tikzpicture}
\end{figure}

We also have another counterexample where both graphs are connected.

\smallskip 

\textbf{Outline of the paper.} In Section~\ref{sec:examples}, we begin with the simplest non-trivial case where $H \cong P_2$ and progressively generalize our construction. In Section~\ref{sec:general-shadow} we generalize further to $H$-subgraphs that represent multiple recoloring sequences, ending with the fully general construction for arbitrary $H$. Section~\ref{sec:hypercube} applies this construction to the special case of counting hypercube graphs in coloring graphs of trees. In Section~\ref{sec:gen-deg-seq}, we investigate how counting hypercube graphs in coloring graphs of trees relates to the generalized degree sequence of a tree. Finally, in Section~\ref{sec:counterexample}, we present two different counterexamples to \cite{AKLR25}*{Conjecture 5.2}.

\textbf{Acknowledgments.} I gratefully acknowledge the support of Santa Clara University through a research stipend I received during Summer 2024, which made this work possible. This funding was administered via the REAL (Research Experience for Advanced Learning) program. I also wish to acknowledge the honor of receiving the Rick Scott Memorial Scholarship. I am also grateful to the anonymous referee for their valuable feedback that improved the exposition.

\section{Shadow Graphs for Basic Subgraphs}\label{sec:examples}

In this section we introduce shadow graphs, the main constructive tool of this paper. We demonstrate their use in the simplest possible case, when $H=K_2$. We extend this construction to more complex graphs including cliques and Cartesian products of cliques, the latter of which is used in Section~\ref{sec:hypercube}. The final construction counts induced hexagons, which can represent two possible sequences of recolorings. For disjoint subsets $U, W\subseteq V(G)$, we use the shorthand $U\times W$ to mean the set of all edges $uw$ for $u\in U$ and $w\in W$. For a set of vertices $V$, we use the shorthand $(V)_k$ to mean the set of all $k$-tuples $(v_1,\dots,v_k)$ of distinct vertices in $V$.

\subsection{Shadow graphs}

A shadow graph, $G'$ is a specific kind of graph constructed from $G$ to contain many induced subgraphs isomorphic to $G$, such that the induced subgraphs overlap at some set of vertices. This allows one $k$-coloring of $G'$ to induce many $k$-colorings of $G$ each only differing from each other at a restricted subset of $V(G)$. Recall that our goal is to compute $\pi_G^{(H)}(k)$, which is the number of induced subgraphs of $\mathcal{C}_k(G)$ isomorphic to $H$. By definition, a subgraph of $\mathcal{C}_k(G)$ is precisely a set of $k$-colorings of $G$ with some pairwise conditions on the number of vertices at which they differ. Since shadow graphs are constructed such that their proper $k$-colorings contain a set of $k$-colorings of $G$ with pairwise conditions on the number of vertices at which they differ, the problem of counting subgraphs of $\mathcal{C}_k(G)$ can be translated to counting colorings of an appropriate shadow graph of $G$.

\begin{defin}
    For a graph $G$, some tuple of vertices $\mathbf{x}=(v_1,\dots,v_d)\in (V(G))_d$, and corresponding integers $2\le r_1,\dots,r_d$, a \emph{shadow graph} $G_\mathbf{x}$ has vertex set $V(G)\setminus \mathbf{x} \cup \{v_{i,j}:i \le d, \; j\le r_i\}$. For any two vertices $w_1,w_2\in V(G)\setminus \mathbf{x}$, $w_1w_2\in E(G_\mathbf{x})$ if and only if $w_1w_2\in E(G)$. For a vertex $v_{i,j}\in V(G_\mathbf{x})$, and a vertex $w\in V(G)\setminus \mathbf{x}$, $v_{i,j}w\in E(G_\mathbf{x})$ if and only if $v_iw\in E(G)$. For any $v_{i,j},v_{i,k}\in V(G_\mathbf{x})$, $v_{i,j}v_{i,k}\in E(G_\mathbf{x})$, that is, the $v_{i,j}$ form a \emph{shadow clique} for any fixed $i$. For a pair of vertices $v_{i,j},v_{k,l}\in V(G_\mathbf{x})$, $v_{i,j}v_{k,l}\notin E(G_\mathbf{x})$ if $v_iv_k \notin E(G)$.
\end{defin}
Below is an example of a shadow graph $G_\mathbf{x}$ of the given graph $G$ where $\mathbf{x}=(v_1,v_2)$, $r_1=2$, and $r_2=2$. The purple dotted edges indicate edges that can either be present or absent in $G_\mathbf{x}$:

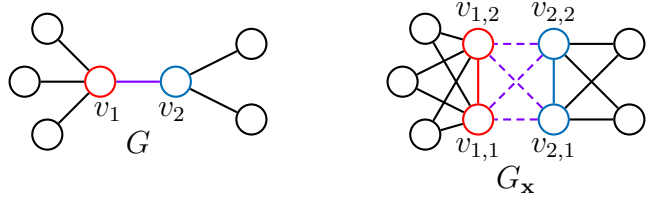
\begin{figure}[ht]
	\centering
	\begin{tikzpicture}
		\node[vertex] (0) at (-1, 0.5) {};
		\node[vertex] (1) at (-1, -0.5) {};
		\node[blue_vertex] (2) at (-2, -0) {};
		\node[] (16) at (-2.9, -0.4) {$v_1$};
		\node[] (17) at (-2.05, -0.4) {$v_2$};
		\node[red_vertex] (3) at (-3, 0) {};
		\node[vertex] (4) at (-3.7, 0.7) {};
		\node[vertex] (5) at (-4, -0) {};
		\node[vertex] (6) at (-3.7, -0.7) {};
		\node[vertex] (7) at (1, 0) {};
		\node[vertex] (8) at (1.3, 0.7) {};
		\node[vertex] (9) at (1.3, -0.7) {};
		\node[red_vertex] (10) at (2, -0.5) {};
		\node[red_vertex] (11) at (2, 0.5) {};
		\node[] (18) at (2, -0.9) {$v_{1,1}$};
		\node[] (19) at (2, 0.9) {$v_{1,2}$};
		\node[] (22) at (2.5, -1.3) {$G_\mathbf{x}$};
		\node[] (23) at (-2.5, -0.8) {$G$};
		\node[blue_vertex] (12) at (3, -0.5) {};
		\node[blue_vertex] (13) at (3, 0.5) {};
		\node[] (20) at (3, -0.9) {$v_{2,1}$};
		\node[] (21) at (3, 0.9) {$v_{2,2}$};
		\node[vertex] (14) at (4, 0.5) {};
		\node[vertex] (15) at (4, -0.5) {};
		\draw[edge] (2) -- (0);
		\draw[edge] (2) -- (1);
		\draw[purple_edge] (3) -- (2);
		\draw[edge] (4) -- (3);
		\draw[edge] (5) -- (3);
		\draw[edge] (6) -- (3);
		\draw[edge] (10) -- (7);
		\draw[edge] (10) -- (8);
		\draw[edge] (10) -- (9);
		\draw[edge] (11) -- (7);
		\draw[edge] (11) -- (8);
		\draw[edge] (11) -- (9);
		\draw[red_edge] (11) -- (10);
		\draw[purple_dotted_edge] (12) -- (10);
		\draw[purple_dotted_edge] (12) -- (11);
		\draw[purple_dotted_edge] (13) -- (10);
		\draw[purple_dotted_edge] (13) -- (11);
		\draw[blue_edge] (13) -- (12);
		\draw[edge] (14) -- (12);
		\draw[edge] (14) -- (13);
		\draw[edge] (15) -- (12);
		\draw[edge] (15) -- (13);
	\end{tikzpicture}
\end{figure}
The presence or absence of edges $v_{i,j}v_{k,l}\in E(G_\mathbf{x})$ when $v_iv_k\in E(G)$ is determined according to the structure of $H$ by an algorithm described in Section~\ref{sec:general-shadow}. For a given choice of such edges, it may be the case that for some $j_1,\dots,j_d$, that $v_{i,j_i}v_{k,j_k}\in E(G_\mathbf{x})$ for all $v_iv_k\in E(G)$. This means that the induced subgraph of $G_\mathbf{x}$ on $V(G)\setminus S \cup \{v_{i,j_i} : i\le d\}$ is isomorphic to $G$. In such a case, any $k$-coloring $c$ of $G_\mathbf{x}$ will induce a $k$-coloring $\phi$ of $G$ where $\phi(v_i)=c(v_{i,j_i})$. If multiple such choices of the $j_i$ exist, $c$ induces multiple colorings of $G$ each agreeing at all $w\in V(G)\setminus S$, and possibly at some of the vertices in $S$ depending on where different choices overlap.

\subsection{Chromatic Pairs Polynomial}
The chromatic pairs polynomial is the simplest example of a subgraph counting polynomial aside from the well-studied chromatic polynomial. The chromatic pairs polynomial, denoted $\pi_G^{(P_2)}(k)$, is defined as the function that counts the number of edges (equivalently the number of induced copies of $P_2$) in $\mathcal{C}_k(G)$. By definition, an edge in $\mathcal{C}_k(G)$ represents a pair of colorings $\phi_1,\phi_2$ of $G$ differing only at a single vertex. Given such a vertex, $v$, we can construct a shadow graph containing two induced copies of $G$ overlapping at all vertices except $v$. Such a shadow graph is constructed as follows:

For a vertex $v \in V(G)$, the $P_2$-shadow graph of $G$ on $v$, denoted $G_v^{(P_2)}$ (or simply $G_v$ by abuse of notation), is the graph obtained from $G$ where $V(G_v) = (V(G) \setminus \{v\}) \cup \{v_1, v_2\}$ and 
$E(G_v) = E(G - v) \cup \{v_1v_2\} \cup (N(v)\times\{v_1,v_2\})$.
Note that $G_{v}$ has one more vertex and $\deg(v)+1$ more edges than $G$.
Below is an example of the $P_2$-shadow graph of a given graph $G$ on a vertex $v$.
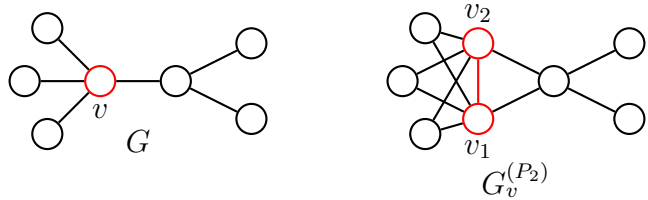
\begin{figure}[ht]
	\centering
	\begin{tikzpicture}
		\node[vertex] (0) at (-1, 0.5) {};
		\node[vertex] (1) at (-1, -0.5) {};
		\node[vertex] (2) at (-2, -0) {};
		\node[red_vertex] (3) at (-3, -0) {};
		\node[] (16) at (-3, -0.4) {$v$};
		\node[vertex] (4) at (-3.7, 0.7) {};
		\node[vertex] (5) at (-4, -0) {};
		\node[vertex] (6) at (-3.7, -0.7) {};
		\node[vertex] (7) at (1, 0) {};
		\node[vertex] (8) at (1.3, 0.7) {};
		\node[vertex] (9) at (1.3, -0.7) {};
		\node[red_vertex] (10) at (2, 0.5) {};
		\node[red_vertex] (11) at (2, -0.5) {};
		\node[] (15) at (2, 0.9) {$v_2$};
		\node[] (17) at (2, -0.9) {$v_1$};
		\node[vertex] (12) at (3, -0) {};
		\node[vertex] (13) at (4, 0.5) {};
		\node[vertex] (14) at (4, -0.5) {};
		\node[] (18) at (2.5, -1.3) {$G_v^{(P_2)}$};
		\node[] (19) at (-2.5, -0.8) {$G$};
		\draw[edge] (2) -- (0);
		\draw[edge] (2) -- (1);
		\draw[edge] (3) -- (2);
		\draw[edge] (4) -- (3);
		\draw[edge] (5) -- (3);
		\draw[edge] (6) -- (3);
		\draw[edge] (10) -- (7);
		\draw[edge] (10) -- (8);
		\draw[edge] (10) -- (9);
		\draw[edge] (11) -- (7);
		\draw[edge] (11) -- (8);
		\draw[edge] (11) -- (9);
		\draw[red_edge] (11) -- (10);
		\draw[edge] (12) -- (10);
		\draw[edge] (12) -- (11);
		\draw[edge] (13) -- (12);
		\draw[edge] (14) -- (12);
	\end{tikzpicture}
\end{figure}
\begin{lemma}\label{lem:pairs}
The chromatic pairs polynomial of $G$ is given by
$$
\pi_G^{(P_2)}(k) = \frac{1}{2} \sum_{v \in V} \pi_{G_v}(k).
$$
\end{lemma}

\begin{proof}
For any proper coloring $c$ of $G_v$, define colorings $\phi_1, \phi_2$ of $G$, where $\phi_1(u)=\phi_2(u)=c(u)$ for all $u\neq v$, and $\phi_i(v)=c(v_i)$. Since $v_1v_2 \in E(G_v)$, we have $c(v_1) \neq c(v_2)$, so $\phi_1$ and $\phi_2$ differ only at $v$, corresponding to an edge in $\mathcal{C}_k(G)$. These colorings are proper, as the induced subgraph of $G_v$ on $V(G_v)\setminus \{v_2\}$ is isomorphic to $G$, and the restriction of $c$ to this subgraph is $\phi_1$. Similarly, the induced subgraph of $G_v$ on $V(G_v)\setminus \{v_1\}$ is isomorphic to $G$, and the restriction of $c$ to this subgraph is $\phi_1$. Conversely, for any edge $\phi_1\phi_2\in E(\mathcal{C}_k(G))$, we have that $\phi_1(u)=\phi_2(u)$ for all $u\ne v$ for some $v\in V(G)$. Define $c$ as the coloring of $G_v$ where $c(u)=\phi_1(u)=\phi_2(u)$ for all $u\ne v$, $c(v_1)=\phi_1(v)$, and $c(v_2)=\phi_2(v)$. For any $w\in N(v)$, $c(v_1)\ne c(w)$, because $\phi_1$ is a proper coloring, therefore the edges $v_1w$ do not violate the proper coloring condition. The same holds for edges $v_2w$, due to the properness of $\phi_2$. Since $\phi_1$ and $\phi_2$ are assumed to differ at $v$, $c(v_1)\ne c(v_2)$, so the edge $v_1v_2\in E(G_v)$ does not violate the proper coloring condition. Therefore $c$ is a proper coloring of $G_v$. This correspondence is 2-to-1, as swapping $c(v_1)$ and $c(v_2)$ yields the same edge. The formula follows by summing over all vertices.
\end{proof}

\subsection{\texorpdfstring{$r$}{r}-Clique Polynomials}
The construction of $G_v^{(P_2)}$ can be extended to count occurrences of cliques as subgraphs of $\mathcal{C}_k(G)$. It is known that an occurrence of $K_r$ as a subgraph of $\mathcal{C}_k(G)$ must signify a single vertex changing between $r$ available colors, while the other vertices remain the same color. That is, if $K_r\cong H \subseteq \mathcal{C}_k(G)$, and $V(H)=\{\phi_1,\dots,\phi_r\}$, then $\phi_1,\dots,\phi_r$ all differ only at a single vertex. Given such a vertex, $v$, we can construct a shadow graph containing $r$ induced copies of $G$ each overlapping at all vertices except $v$ as follows:

For a vertex $v \in V(G)$, the $K_r$-shadow graph of $G$ on $v$, denoted $G_v^{(K_r)}$ (or simply $G_v$ when context is clear), is the graph obtained from $G$ where:
$V(G_v) = (V(G) \setminus \{v\}) \cup \{v_1,\ldots,v_r\}$ and
$E(G_v) = E(G - v) \cup \{v_iv_j : 1 \leq i < j \leq r\} \cup (N(v)\times \{v_1,\ldots,v_r\})$. Below is an example of the $K_5$-shadow graph of a given graph $G$ on a vertex $v$:
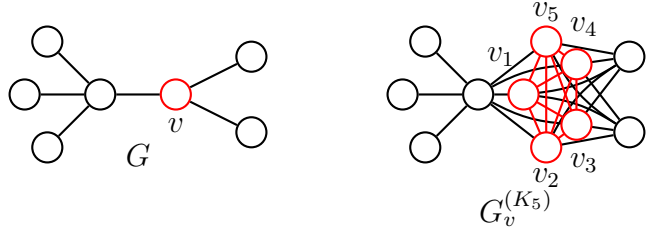
\begin{figure}[ht]
	\centering
	\begin{tikzpicture}
		\node[vertex] (0) at (-1, 0.5) {};
		\node[vertex] (1) at (-1, -0.5) {};
		\node[red_vertex] (2) at (-2, 0) {};
		\node[] (16) at (-2, -0.4) {$v$};
		\node[vertex] (3) at (-3, 0) {};
		\node[vertex] (4) at (-3.7, 0.7) {};
		\node[vertex] (5) at (-3.7, -0.7) {};
		\node[vertex] (6) at (-4, -0) {};

		\node[vertex] (7) at (1, -0) {};
		\node[vertex] (8) at (1.3, 0.7) {};
		\node[vertex] (9) at (1.3, -0.7) {};
		\node[vertex] (10) at (2, -0) {};
		\node[vertex] (11) at (4, 0.5) {};
		\node[vertex] (12) at (4, -0.5) {};
		\node[] (13) at (2.6, 0) {};
		\node[] (14) at (2.9, 0.7) {};
		\node[] (15) at (2.9, -0.7) {};
		\node[] (16) at (3.3, -0.4) {};
		\node[] (17) at (3.3, 0.4) {};
		\node[] (23) at (2.5, -1.5) {$G_v^{(K_5)}$};
		\node[] (24) at (-2.5, -0.8) {$G$};
		\node[] (25) at (2.3, 0.5) {$v_1$};
		\node[] (26) at (2.9, -1.1) {$v_2$};
		\node[] (27) at (3.4, -0.9) {$v_3$};
		\node[] (28) at (3.4, 0.9) {$v_4$};
		\node[] (29) at (2.9, 1.1) {$v_5$};
		\draw[edge] (2) -- (0);
		\draw[edge] (2) -- (1);
		\draw[edge] (3) -- (2);
		\draw[edge] (4) -- (3);
		\draw[edge] (5) -- (3);
		\draw[edge] (6) -- (3);
		\draw[edge] (10) -- (7);
		\draw[edge] (10) -- (8);
		\draw[edge] (10) -- (9);
		\draw[edge] (13) -- (10);
		\draw[edge, bend right = 13] (13) to (11);
		\draw[edge, bend left = 13] (13) to (12);
		\draw[edge] (14) -- (10);
		\draw[edge] (14) -- (11);
		\draw[edge, bend right = 15] (14) to (12);
		\draw[edge] (15) -- (10);
		\draw[edge, bend left = 15] (15) to (11);
		\draw[edge] (15) -- (12);
		\draw[edge, bend left = 12] (16) to (10);
		\draw[edge] (16) -- (11);
		\draw[edge] (16) -- (12);
		\draw[edge, bend right = 12] (17) to (10);
		\draw[edge] (17) -- (11);
		\draw[edge] (17) -- (12);
		\draw[red_edge] (14) -- (13);
		\draw[red_edge] (15) -- (13);
		\draw[red_edge] (15) -- (14);
		\draw[red_edge] (16) -- (13);
		\draw[red_edge] (16) -- (14);
		\draw[red_edge] (16) -- (15);
		\draw[red_edge] (17) -- (13);
		\draw[red_edge] (17) -- (14);
		\draw[red_edge] (17) -- (15);
		\draw[red_edge] (17) -- (16);
		\node[red_vertex] (18) at (2.6, 0) {};
		\node[red_vertex] (19) at (2.9, 0.7) {};
		\node[red_vertex] (20) at (2.9, -0.7) {};
		\node[red_vertex] (21) at (3.3, -0.4) {};
		\node[red_vertex] (22) at (3.3, 0.4) {};
	\end{tikzpicture}
\end{figure}

\begin{lemma}\label{lem:cliques}
The $r$-clique polynomial of $G$ is given by
$$
\pi_G^{(K_r)}(k) = \frac{1}{r!} \sum_{v \in V} \pi_{G_v}(k).
$$
\end{lemma}

\begin{proof}
For any proper coloring $c$ of $G_v$, define colorings $\phi_1,\ldots,\phi_r$ of $G$, where $\phi_i(u)=c(u)$ for all $u\neq v$, and $\phi_i(v)=c(v_i)$. Since $v_iv_j \in E(G_v)$ for all $i\neq j$, these $r$ colorings all differ from each other only at $v$, forming a clique in $\mathcal{C}_k(G)$. Conversely, for an $r$-clique $K_r\cong H\subseteq \mathcal{C}_k(G)$, where $V(H)=\{\phi_1,\dots,\phi_r\}$, we know that $\phi_1,\dots,\phi_r$ all differ at a single vertex, which we will call $v$. We can construct a coloring $c$ of $G_v$ where $c(w)=\phi_1(w)=\dots=\phi_r(w)$ for all $w \in V(G)\setminus \{v\}$, and $c(v_i)=\phi_i(v)$. By the properness of $\phi_i$, $v_i$ does not share a color with any $w\in N(v)$, and since the $\phi_i$ were assumed to be distinct colorings, $c(v_i)\ne c(v_j)$ for any $i,j\le r$. This correspondence is $r!$-to-1, as permuting the colors $\{c(v_i)\}_{i=1}^r$ yields the same clique, and for any clique there are $r!$ ways of picking the labels $\phi_i$, each yielding distinct colorings of $G_v$. The formula follows by summing over $v \in V(G)$.
\end{proof}

\subsection{Polynomials Counting Cartesian Products of Cliques}\label{subsec:cartesian}
The construction of $G_v^{(K_r)}$ can be extended further still to count a family of subgraphs that arises from multiple vertices independently switching colors. Suppose we have such a set of vertices $\{v_1,\dots,v_d\}\subseteq V(G)$, such that across some set of colorings in $\mathcal{C}_k(G)$, each vertex $v_i$ picks its color out of a color palette $\{c_{i,1},\dots,c_{i,r_i}\}$. Consider the subgraph induced by all such color choices, assuming the choices can be made independently (with the left over vertices remaining the same color). Any such coloring $\phi$ can be represented by a tuple $(j_1,\dots,j_d)\in \{1,\dots,r_1\}\times\dots\times\{1,\dots,r_d\}$, where $\phi(v_i)=c_{i,j_i}$. Two colorings are adjacent if and only if their respective tuples differ at a single index. This is the construction for the graph $H=\mathop{\square}_{i=1}^d K_{r_i}$, a Cartesian product of complete graphs. Conversely, it is known that an induced Cartesian product of cliques can only arise as a subgraph of $\mathcal{C}_k(G)$ by some set of $d$ vertices independently switching between their respective $r_i$ colors in this way. To construct a shadow graph to count occurrences of such a behavior, we do the following:

Given positive integers $r_1,\dots, r_d$, construct a shadow graph $G_\mathbf{x}$ of $G$ on a $d$-tuple of vertices $\mathbf{x}=(v_1,\dots,v_d)$ by replacing each vertex $v_i\in \mathbf{x}$ with an $r_i$-clique of vertices $v_{i,1},\dots,v_{i,r_i}$. Add edges between each $v_{i,j}$ and all neighbors of $v_i$ in $G$, excluding those neighbors also in $\mathbf{x}$. For any pair of vertices $v_i,v_k\in \mathbf{x}$ where $v_iv_k\in E(G)$, add all edges $v_{i,j}v_{k.l}$ between their corresponding shadow cliques. Below is such an example of such a construction with $r_1=3,r_2=2,r_3=3$, and $\mathbf{x}=(v_1,v_2,v_3)$ on a graph $G\cong P_5$:
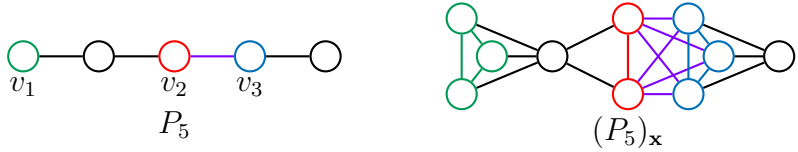
\begin{figure}[ht]
	\centering
	\begin{tikzpicture}
		\node[vertex] (0) at (-1, 0) {};
		\node[blue_vertex] (1) at (-2, 0) {};
		\node[] (15) at (-2, -0.4) {$v_3$};
		\node[red_vertex] (2) at (-3, 0) {};
		\node[] (16) at (-3, -0.4) {$v_2$};
		\node[vertex] (3) at (-4, 0) {};
		\node[green_vertex] (4) at (-5, 0) {};
		\node[] (17) at (-5, -0.4) {$v_1$};
		\node[green_vertex] (5) at (0.8, 0.5) {};
		\node[green_vertex] (6) at (0.8, -0.5) {};
		\node[green_vertex] (7) at (1.2, -0) {};
		\node[vertex] (8) at (2, -0) {};
		\node[red_vertex] (9) at (3, 0.5) {};
		\node[red_vertex] (10) at (3, -0.5) {};
		\node[blue_vertex] (11) at (3.8, 0.5) {};
		\node[blue_vertex] (12) at (3.8, -0.5) {};
		\node[blue_vertex] (13) at (4.2, 0) {};
		\node[vertex] (14) at (5, -0) {};
		\draw[edge] (1) -- (0);
		\draw[purple_edge] (2) -- (1);
		\draw[edge] (3) -- (2);
		\draw[edge] (4) -- (3);
		\draw[green_edge] (6) -- (5);
		\draw[green_edge] (7) -- (5);
		\draw[green_edge] (7) -- (6);
		\draw[edge] (8) -- (5);
		\draw[edge] (8) -- (6);
		\draw[edge] (8) -- (7);
		\draw[edge] (9) -- (8);
		\draw[edge] (10) -- (8);
		\draw[red_edge] (10) -- (9);
		\draw[purple_edge] (11) -- (9);
		\draw[purple_edge] (11) -- (10);
		\draw[purple_edge] (12) -- (9);
		\draw[purple_edge] (12) -- (10);
		\draw[purple_edge] (13) -- (9);
		\draw[purple_edge] (13) -- (10);
		\draw[blue_edge] (12) -- (11);
		\draw[blue_edge] (13) -- (11);
		\draw[blue_edge] (13) -- (12);
		\draw[edge] (14) -- (11);
		\draw[edge] (14) -- (12);
		\draw[edge] (14) -- (13);
		\node[] (18) at (-3, -0.9) {$P_5$};
		\node[] (19) at (3, -1) {$(P_5)_\mathbf{x}$};
	\end{tikzpicture}
\end{figure}

\begin{lemma}
Let $H = \mathop{\square}_{i=1}^d K_{r_i}$ be the Cartesian product. Then
$$
\pi_G^{(H)}(k) = \frac{1}{\prod_{i=1}^d r_i! \cdot \prod_{j=1}^{\ell} m_j!} \sum_{\mathbf{x} \in (V(G))_d} \pi_{G_\mathbf{x}}(k),
$$
\end{lemma}

\begin{proof}
For any proper coloring $c$ of $G_\mathbf{x}$, define colorings $\phi_{(j_1,\dots,j_d)}$ for all $(j_1,\dots,j_d)\in \{1,\dots,r_1\}\times\dots\times\{1,\dots,r_d\}$, wherein $\phi_{(j_1,\dots,j_d)}(v_i)=c(v_{i,j_i})$. Across all tuples, the vertices in $\mathbf{x}$ take on all possible color combinations within the color palette contained within each vertex's shadow clique. Because $v_{i,j_i}v_{k,j_k} \in E(G_\mathbf{x})$ whenever $v_iv_k\in E(G)$, we must have that $\phi_{(j_1,\dots,j_d)}(v_i)\ne \phi_{(j_1,\dots,j_d)}(v_k)$, therefore $\phi_{(j_1,\dots,j_d)}$ is a proper coloring. Conversely, for some $\mathop{\square}_{i=1}^d K_{r_i}\cong H \subseteq \mathcal{C}_k(G)$, we know that across all $k$-colorings of $G$ in $V(H)$, there must be some set of vertices $v_1,\dots,v_d$, with $v_i$ switching between $r_i$ different colors $c_{i,1},\dots, c_{i,r_i}$ independently of the color choice of the other vertices. The subgraph $H$ represents all such possible color choices, which can be indexed as $\phi_{(j_1,\dots,j_d)}$, where $\phi_{(j_1,\dots,j_d)}(v_i)=c_{i,j_i}$. Let $\mathbf{x}=(v_1,\dots,v_d)$, and $S=\{v_{i,j} : i\le d,\; j\le r_i\}$. We can construct a coloring $c$ of $G_\mathbf{x}$ where $c(w)=\phi(w)$ for all $\phi \in V(H)$ when $w\notin S$, and where $c(v_{i,j})=c_{i,j}$. The edges $w_1w_2\in E(G_\mathbf{x})$ where $w_1,w_2\notin S$ do not violate the proper coloring condition as $w_1w_2\in \phi$ for all $\phi \in V(H)$, and $V(H)$ consists of proper colorings. The edges $w v_{i,j} \in E(G_\mathbf{x})$ where $w\notin S$ do not violate the proper coloring condition, as $\phi_{(j_1,\dots,j_d)}(v_i)\ne \phi_{(j_1,\dots,j_d)}(w)$. The edges $v_{i,j}v_{i,l}\in E(G_\mathbf{x})$ do not violate the proper coloring condition as the shadow clique of $v_i$ is colored with the color palette of $v_i$ which by definition does not contain repeat colors. The edges $v_{i,j}v_{k,l}\in E(G_\mathbf{x})$ do not violate the proper coloring condition as distinct vertices $v_i,v_k$ are assumed to be able to switch between colors in their respective color palettes independently of the other color choices. Therefore adjacent vertices have non-overlapping color palettes within $H$.

The factor $\prod_{i=1}^d r_i!$ accounts for the permutations of colors among the shadows of each $v_i$. The additional division by $\prod_{j} m_j!$ corrects for over-counting when identical $r_i$ values occur multiple times. For instance, if $r_1=r_2$, then $G_{(v_1, v_2, \ldots)}\cong G_{(v_2, v_1, \ldots)}$, and they will count the same Cartesian products of cliques, despite being generated by different $d$-tuples.

Summing over all ordered tuples and dividing by these factors counts instances of $\mathop{\square}_{i=1}^d K_{r_i}$.
\end{proof}

\subsection{Calculating \texorpdfstring{$\pi_G^{(C_6)}(k)$}{π\textasciicircum{}(C₆)(k)}}

So far all of the subgraphs we have been counting can only appear in $\mathcal{C}_k(G)$ by a single, predictable pattern of vertex recolorings. An edge can only appear as a result of one vertex changing between 2 colors. A clique can only appear as a result of one vertex changing between multiple colors. A Cartesian product of cliques can only appear from a set of vertices independently changing between a set of colors. For general $H$, this is not always the case. For instance, it is established in \cite{ABFR17} that an induced copy of $C_6$ can appear in one of two ways, either by two vertices switching between three colors, or by three vertices switching between two colors. For completeness, we include a short proof:
\begin{proof}
    If only one vertex changes, the sequence degenerates into a $K_6$ as repeated color changes on the same vertex would make every coloring in the sequence adjacent. 

    If two vertices change color, neither can change color more than 3 times, as that would mean one vertex changes its color multiple times consecutively, inducing a triangle. Since $C_6$ does not contain a triangle, each can change at most 3 times in total. Since there are only 2, they must each change 3 times exactly.

    If three vertices change color, each must change color at least twice, as traversing all edges of the $C_6$ must end at the same coloring as the start. Since there are 3 vertices, each must change exactly 2 times.
    
    If more than three vertices change color at some point, they all must change back to their original color at some point due to the cyclic nature of the hexagon. This would require more than 6 edges, a contradiction.
\end{proof}

We begin by counting occurrences of $C_6\cong H\subseteq \mathcal{C}_k(G)$, where only 2 vertices $v_1,v_2\in V(G)$ change color across all edges in $H$. As established, each vertex has a color palette of three colors, $\{c_{1,1},c_{1,2},c_{1,3}\}$ and $\{c_{2,1},c_{2,2},c_{2,3}\}$ respectively. Any coloring $\phi$ in $V(H)$ can then be identified with a tuple $(i,j)$ signifying that $\phi(v_1)=c_{1,i}$ and $\phi(v_2)=c_{2,j}$. The only valid recoloring sequence up to relabeling of $v_1$, $v_2$, and their respective color palettes is
$$
(1,1)\to (2,1)\to(2,2)\to(3,2)\to(3,3)\to(1,3)\to(1,1)
$$

Each pair $(i,j)$ represents a coloring state in $H$. Adjacent colorings in the sequence differ at exactly one vertex, ensuring proper adjacency. Moreover, any two colorings corresponding to non-adjacent vertices in $H$ do not differ by the color of only one vertex, which ensures the subgraph is induced.

To count these hexagons, we construct a shadow graph $G_{(v_1,v_2)}$ by the following construction on $G$ with vertices $v_1$ and $v_2$ selected:
\begin{itemize}
    \item Replace $v_1$ and $v_2$ each with a triangle containing vertices $v_{1,1},v_{1,2},v_{1,3}$ and $v_{2,1},v_{2,2},v_{2,3}$
    \item Add edges between $v_{1,1},v_{1,2},v_{1,3}$ and all vertices in $N(v_1)\setminus \{v_2\}$, and add edges between $v_{2,1},v_{2,2},v_{2,3}$ and $N(v_2)\setminus \{v_1\}$
    \item If $v_1v_2\in E(G)$, connect $v_{1,1},v_{1,2},v_{1,3},v_{2,1},v_{2,2},v_{2,3}$ such that each state in the recoloring sequence, $(i,j)$ is assigned an edge $v_{1,i} v_{2,j}$. For this recoloring sequence, we add edges $v_{1,1}v_{2,1},v_{1,2}v_{2,1},v_{1,2}v_{2,2},v_{1,3}v_{2,2},v_{1,3}v_{2,3},v_{1,1}v_{2,3}$
\end{itemize}
An example of such a $G_{(v_1,v_2)}$ is shown below for $G=P_5$ and $v_1,v_2$ given:
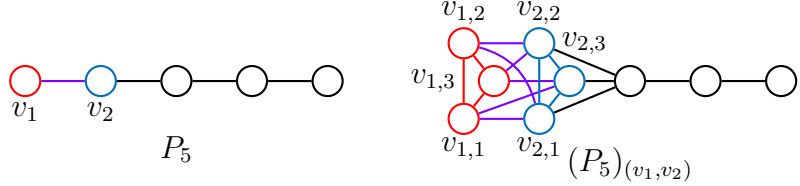
\begin{figure}[ht]
    \centering
    \begin{tikzpicture}
        \node[vertex] (0) at (0, 0) {};
        \node[red_vertex] (1) at (-4, 0) {};
        \node[] (21) at (-4, -0.4) {$v_1$};
        \node[blue_vertex] (2) at (-3, 0) {};
        \node[] (22) at (-3, -0.4) {$v_2$};
        \node[vertex] (3) at (-2, 0) {};
        \node[vertex] (4) at (-1, 0) {};
        \node[vertex] (5) at (6, -0) {};
        \node[red_vertex] (6) at (1.8, 0.5) {};
        \node[] (26) at (1.8, 0.9) {$v_{1,2}$};
        \node[red_vertex] (7) at (1.8, -0.5) {};
        \node[] (27) at (1.8, -0.9) {$v_{1,1}$};
        \node[red_vertex] (8) at (2.2, 0) {};
        \node[] (28) at (1.4, 0) {$v_{1,3}$};
        \node[blue_vertex] (9) at (2.8, 0.5) {};
        \node[] (29) at (2.8, 0.9) {$v_{2,2}$};
        \node[blue_vertex] (10) at (3.2, 0) {};
        \node[] (30) at (2.8, -0.9) {$v_{2,1}$};
        \node[blue_vertex] (11) at (2.8, -0.5) {};
        \node[] (31) at (3.4, 0.5) {$v_{2,3}$};
        \node[vertex] (12) at (4, 0) {};
        \node[vertex] (13) at (5, -0) {};
        \draw[edge] (4) -- (0);
        \draw[purple_edge] (2) -- (1);
        \draw[edge] (3) -- (2);
        \draw[edge] (4) -- (3);
        \draw[edge] (13) -- (5);
        \draw[red_edge] (7) -- (6);
        \draw[red_edge] (8) -- (6);
        \draw[red_edge] (8) -- (7);
        \draw[purple_edge] (9) -- (6);
        \draw[purple_edge] (9) -- (8);
        \draw[purple_edge] (10) -- (7);
        \draw[purple_edge] (10) -- (8);
        \draw[blue_edge] (10) -- (9);
        \draw[purple_edge, bend right = 30] (11) to (6);
        \draw[purple_edge] (11) -- (7);
        \draw[blue_edge] (11) -- (9);
        \draw[blue_edge] (11) -- (10);
        \draw[edge] (12) -- (9);
        \draw[edge] (12) -- (10);
        \draw[edge] (12) -- (11);
        \draw[edge] (13) -- (12);
		\node[] (40) at (-2, -0.9) {$P_5$};
		\node[] (41) at (4, -1.1) {$(P_5)_{(v_1,v_2)}$};
    \end{tikzpicture}

\end{figure}

These constraints ensure that the chromatic polynomial of $G_{(v_1,v_2)}$ counts only colorings corresponding to a recoloring sequence forming an induced $C_6$. For any $k$-coloring $c$ of $G_{(v_1,v_2)}$, consider the colorings $\phi_{(i,j)}$ for which $\phi_{(i,j)}(v_1)=c(v_{1,i})$, $\phi_{(i,j)}(v_2)=c(v_{2,j})$, and $\phi_{(i,j)}(w)=c(w)$ otherwise. Whenever $(i,j)$ is one of the states in the recoloring sequence, $\phi_{(i,j)}$ is a proper coloring. Edges between pairs of non-shadow vertices do not violate the proper coloring condition, as they inherit the properness of their corresponding edge in $c$. Similarly, edges from $v_1$ or $v_2$ to their non-shadow neighbors do not violate the properness of the coloring, because edges $wv_{1,i}$ (respectively $wv_{2,i}$) exist in $G_{(v_1,v_2)}$. Finally, the edge $v_1v_2$ does not violate the properness of the coloring if it is present, as $v_{1,i}v_{2,j}$ is present for any $(i,j)$ in the recoloring sequence. Conversely, consider some set of colorings $\phi_{(i,j)}$ of $G$ forming an induced $C_6$ in $\mathcal{C}_k(G)$, where $\phi_{(i,j)}(v_1)=c_{1,i}$ and $\phi_{(i,j)}(v_2)=c_{2,j}$ for some pair of vertices $(v_1,v_2)$. Let it also be the case that $v_1,v_2$ and $c_{1,i},c_{2,j}$ are chosen such that for a chosen traversal around the $C_6$, the order at which we traverse the $\phi_{(i,j)}$ matches the earlier recoloring sequence. Given this, we can construct a proper coloring $c$ of $G_{(v_1,v_2)}$. Let $c(w)=\phi_{(i,j)}(w)$ for all $w\notin \{v_1,v_2\}$ for all $\phi_{(i,j)}$. Let $c(v_{1,i})=c_{1,i}$ and $c(v_{2,j})=c_{2,j}$. Since our choice of $v_1,v_2$ and $c_{1,i},c_{2,j}$ were chosen to align with a chosen traversal of the $C_6$ matching the recoloring sequence, all edges $v_{1,i}v_{2,j}$ do not violate the properness of the coloring $c$. The remaining edges inherit their properness from the $\phi_{(i,j)}$ in the same way as in the Cartesian product of cliques. Across all $k$-colorings of $G_{(v_1,v_2)}$ for all ordered pairs of distinct vertices $(v_1,v_2)$, the correspondence is $|\operatorname{Aut}(C_6)| = 12$-to-1, as a specific coloring of a specific $G_{(v_1,v_2)}$ corresponds to a specific traversal of an induced $C_6$ in $\mathcal{C}_k(G)$, of which each $C_6$ has exactly 12. Summing over all ordered pairs of vertices $(v_1, v_2)$ in $G$ and dividing by 12 gives a polynomial equation counting induced hexagons in $G$ which arise from exactly 2 vertices changing color:
$$
\frac{1}{12} \sum_{(v_1,v_2) \in (V(G))_2} \pi_{G_{(v_1,v_2)}}(k).
$$

When three vertices $v_1$, $v_2$, and $v_3$ all change color across vertices of an induced $C_6\cong H \subseteq \mathcal{C}_k(G)$, they have color palettes of $\{c_{1,1},c_{1,2}\}$, $\{c_{2,1},c_{2,2}\}$, and $\{c_{3,1},c_{3,2}\}$, respectively. Each coloring $\phi$ in $V(H)$ be associated a tuple $(i, j, k)$, representing that $\phi(v_1)=c_{1,i}$, $\phi(v_2)=c_{2,j}$, $\phi(v_3)=c_{3,k}$

The only valid set of such tuples (up to relabeling $v_1,v_2,v_3$, and $c_{1,1},c_{1,2},c_{2,1},c_{2,2},c_{3,1},c_{3,2}$) representing an induced hexagon in which 3 vertices all change color is, in order of a traversal around the $C_6$:
$$
(1,1,1) \to (2,1,1) \to (2,2,1) \to (2,2,2) \to (1,2,2) \to (1,1,2) \to (1,1,1).
$$
As before, we construct a shadow graph $G_{(v_1,v_2,v_3)}$ by the following construction on $G$ with vertices $v_1$, $v_2$, and $v_3$ selected:
\begin{itemize}
    \item Replace $v_1,v_2,v_3$ each with a copy of $P_2$ containing vertices $v_{1,1},v_{1,2}$, $v_{2,1},v_{2,2}$, and $v_{3,1},v_{3,2}$.
    \item Add edges between $v_{1,1},v_{1,2}$ and each vertex in $N(v_1)\setminus\{v,w\}$. Likewise with $v_{2,1},v_{2,2}$ and $v_{3,1},v_{3,2}$.
    \item Add edges between $v_{1,1},v_{1,2},v_{2,1},v_{2,2},v_{3,1},v_{3,2}$ according to the coloring states in the sequence. If $v_1v_2\in E(G)$ and the sequence contains a state $(i, j, k)$, then we add an edge $v_{1,i}v_{2,j}$. Likewise if $v_1v_3\in E(G)$ or $v_2v_3\in E(G)$.
\end{itemize}
An example of such a $G_{(v_1,v_2,v_3)}$ is shown below for $G=P_5$ and $v_1,v_2,v_3$ given:
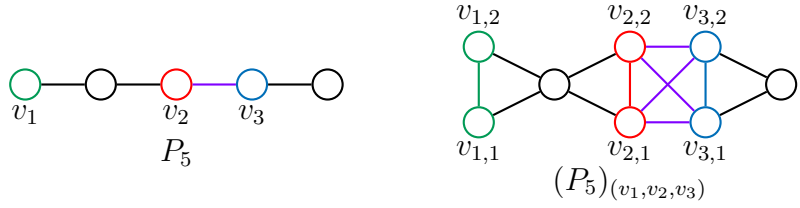
\begin{figure}[ht]
    \centering
    \begin{tikzpicture}
        \node[vertex] (0) at (-1, -0) {};
        \node[blue_vertex] (1) at (-2, 0) {};
        \node[] (21) at (-2, -0.4) {$v_3$};
        \node[red_vertex] (2) at (-3, 0) {};
        \node[] (22) at (-3, -0.4) {$v_2$};
        \node[vertex] (3) at (-4, 0) {};
        \node[green_vertex] (4) at (-5, -0) {};
        \node[] (23) at (-5, -0.4) {$v_1$};
        \node[green_vertex] (5) at (1, 0.5) {};
        \node[] (24) at (1, 0.9) {$v_{1,2}$};
        \node[green_vertex] (6) at (1, -0.5) {};
        \node[] (25) at (1, -0.9) {$v_{1,1}$};
        \node[vertex] (7) at (2, -0) {};
        \node[red_vertex] (8) at (3, 0.5) {};
        \node[] (26) at (3, 0.9) {$v_{2,2}$};
        \node[red_vertex] (9) at (3, -0.5) {};
        \node[] (27) at (3, -0.9) {$v_{2,1}$};
        \node[blue_vertex] (10) at (4, 0.5) {};
        \node[] (28) at (4, 0.9) {$v_{3,2}$};
        \node[blue_vertex] (11) at (4, -0.5) {};
        \node[] (29) at (4, -0.9) {$v_{3,1}$};
        \node[vertex] (12) at (5, -0) {};
        \draw[edge] (1) -- (0);
        \draw[purple_edge] (2) -- (1);
        \draw[edge] (3) -- (2);
        \draw[edge] (4) -- (3);
        \draw[green_edge] (6) -- (5);
        \draw[edge] (7) -- (5);
        \draw[edge] (7) -- (6);
        \draw[edge] (8) -- (7);
        \draw[edge] (9) -- (7);
        \draw[red_edge] (9) -- (8);
        \draw[purple_edge] (10) -- (8);
        \draw[purple_edge] (10) -- (9);
        \draw[purple_edge] (11) -- (8);
        \draw[purple_edge] (11) -- (9);
        \draw[blue_edge] (11) -- (10);
        \draw[edge] (12) -- (10);
        \draw[edge] (12) -- (11);
		\node[] (40) at (-3, -0.9) {$P_5$};
		\node[] (41) at (3, -1.3) {$(P_5)_{(v_1,v_2,v_3)}$};
    \end{tikzpicture}
\end{figure}

As before, the chromatic polynomial of the shadow graph counts valid colorings corresponding to the recoloring sequence. Given any $k$-coloring $c$ of $G_{(v_1,v_2,v_3)}$, we can create $k$-colorings $\phi_{(i,j,k)}$ of $G$ for each tuple $(i,j,k)$ in the earlier recoloring sequence, such that $\phi_{(i,j,k)}(v_1)=c(v_{1,i}),\phi_{(i,j,k)}(v_2)=c(v_{2,j}),\phi_{(i,j,k)}(v_3)=c(v_{3,k})$, and $\phi_{(i,j,k)}(w)=c(w)$ otherwise. The edges $w_1w_2$ for $w_1,w_2\notin \{v_1,v_2,v_3\}$ do not violate the properness of $\phi_{(i,j,k)}$ as they inherit their properness from $c$, as do the edges between $v_1,v_2,v_3$ and their respective neighbors (excluding eachother), as before. The edges $v_1v_2,v_1v_3,v_2v_3$, if they exist, do not violate the properness of $\phi_{(i,j,k)}$, as $(i,j,k)$ is assumed to be in the recoloring sequence, so the necessary edges $v_{1,i}v_{2,j}, v_{1,i}v_{3,k},v_{2,j}v_{3,k}$ were added to $G_{(v_1,v_2,v_3)}$. Conversely, given an induced subgraph $C_6\cong H\subseteq \mathcal{C}_k(G)$ on which 3 vertices change color, by picking a traversal around $H$ we can label these vertices as $v_1,v_2,v_3$ according to the order they change color, and label their color palettes $\{c_{1,1},c_{1,2}\}$, $\{c_{2,1},c_{2,2}\}$, and $\{c_{3,1},c_{3,2}\}$ according to the order the vertices each take on each color (such that $c_{1,1}$ is the initial color of $v_1$ in the traversal, etc). The correspondence is again $|\text{Aut}(C_6)| = 12$-to-1, as each $k$-coloring of some $G_{(v_1,v_2,v_3)}$ corresponds to a specific traversal around the vertices in an induced $C_6$ in $\mathcal{C}_k(G)$. By picking a specific canonical traversal around the vertices of a given $C_6$, all other traversals can be identified with the automorphism of $C_6$ that sends it to the canonical traversal. Summing over all ordered triples $(v_1, v_2, v_3)$ in $G$ and dividing by $12$ gives the contribution of these sequences to $\pi_G^{(C_6)}(k)$. 

In total:

$$
\pi_G^{(C_6)}(k) = \frac{1}{12} \left( \sum_{(v_1,v_2) \in (V(G))_2} \pi_{G_{(v_1,v_2)}}(k) + \sum_{(v_1, v_2, v_3) \in (V(G))_3} \pi_{G_{(v_1,v_2,v_3)}}(k) \right)
$$

\section{General Construction of a Subgraph Counting Polynomial}\label{sec:general-shadow}

In this section, we fully extend and formalize the construction of subgraph counting polynomials introduced in the previous section for an arbitrary subgraph $H\subseteq\mathcal{C}_k(G)$. We extend the ideas of shadow graphs developed in earlier sections to describe how any subgraph-counting polynomial can be constructed.

\subsection{General Construction for Arbitrary Connected Subgraphs}

To compute $\pi_G^{(H)}(k)$ for an arbitrary connected subgraph $H$, we first determine what underlying behaviors on colorings could cause $H$ to arise as an induced subgraph of some $\mathcal{C}_k(G)$. First, we formalize the concept of an underlying behavior by classifying the behavior of a given induced $H$-subgraph of $\mathcal{C}_k(G)$. Then, we determine necessary conditions for a behavior to give rise to an induced $H$-subgraph. We then show that such a behavior does indeed induce a copy of $H$ wherever it appears in $\mathcal{C}_k(G)$. The behaviors will be classified relative to a specific ordered traversal of a spanning tree of $H$

In any induced $H$-subgraph of $\mathcal{C}_k(G)$, we observe that some set of $d$ vertices change color while the rest remain the same. Let us denote the vertices of this $H$-subgraph as $\phi_1, \dots, \phi_{|V(H)|}$, ordered according to our chosen traversal of the spanning tree. From this, we can create a canonical ordered tuple $\mathbf{x}=(v_1,\dots,v_d)\in (V(G))_d$ of the color-changing vertices in $G$. Here, $v_i$ is the $i$th vertex to change from its initial color during our selected traversal. By using a spanning tree, we ensure that each coloring other than $\phi_1$ differs from at least one previous coloring by exactly one vertex, so that the ordering of $\mathbf{x}$ is unambiguous. We can also assign to each vertex $v_i$ a canonical ordered color palette $\mathcal{P}_i=(c_{i,1},\dots,c_{i,r_i})$ such that $c_{i,1}$ is the initial color of $v_i$, and $c_{i,j}$ is the $j$th color that $v_i$ exhibits as we traverse the spanning tree. These orderings allow us to define a map $f\colon V(H) \to \mathbb{N}^d$ of coloring states of $G$, such that $f(\phi_k)=(j_1,\dots,j_d)$, implies that $\phi_k(v_i)=c_{i,j_i}$. In particular, since we have ordered both the vertices and the colors, $f(\phi_1)=(1,1,\dots, 1)$, and $f(\phi_2)=(2,1,\dots, 1)$. We will consider each $f$ to model the underlying behavior that gives rise to an $H$-subgraph. Whether or not a given $f$ models the underlying behavior of a certain $H$-subgraph of $\mathcal{C}_k(G)$ is a question that relies heavily on $G$ and the chosen $H$-subgraph. However, whether or not some $f$ \emph{can} model the underlying behavior of some $H$-subgraph of some coloring graph is a question that relies only on $H$ and our chosen traversal of $H$.

For a fixed ordered traversal $\mathcal{O}=(u_1, u_2,\dots u_{|V(H)|})$ of the vertices in a spanning tree of $H$, a map $f\colon V(H) \to \mathbb{N}^d$ with $d<|V(H)|$ is called a \emph{valid state map} if:
\begin{itemize}
    \item $f$ is injective, ensuring multiple vertices in $H$ are not assigned to the same vertex in $\mathcal{C}_k(G)$. If two vertices $u,u'\in V(H)$ both mapped to the same tuple $(j_1,\dots,j_d)$, this would signify that each color changing vertex $v_i$ is colored $c_{i,j_i}$ at two different vertices in $H$. Since all other vertices are assumed to remain the same color, these colorings would be the same, contradicting them representing distinct vertices in $H$. Since $f$ should represent the general underlying behavior causing $H$ to appear, only injective $f$ should be included.
    \item Two vertices in $H$ have tuples differing at exactly one index if and only if they are adjacent, ensuring $H$ will be an induced subgraph. 
    \item The presence of some $j>1$ at index $i$ of $f(u_k)$ implies $f(u_{k'})$ has $j-1$ in spot $i$ for some $k'<k$, ensuring each index progresses sequentially through a valid ordered color palette. A map in which $v_i$ exhibits the $j$th color in its color palette without first exhibiting the $j-1$th color, cannot represent traversal in which the color palette is ordered correctly.
    \item The presence of $j>1$ in spot $i>1$ of $f(u_k)$ implies $f(u_{k'})$ has $2$ in spot $i-1$ for some $k'<k$, ensuring the order within the tuple by which indices change corresponds to the order of the tuple itself. A map in which $v_i$ changes from its initial color before $v_{i-1}$ cannot represent a valid traversal in which the vertices are ordered correctly.
\end{itemize}
The last two conditions are necessary because we seek only the $f$ that can arise from traversing some $H$ subgraph after having correctly ordered the vertices and their color palettes. Any $f$ that indicates an incorrectly ordered color palette or incorrectly ordered tuple of color changing vertices should be discarded.
The set $\mathcal{F}(H,\mathcal{O})$ of all such $f$ which exist for a given graph $H$ and respect an ordered traversal $\mathcal{O}$ of $H$ is independent of whichever $G$ we choose as our base graph. 

Consider a set of colorings $\phi_1,\dots,\phi_{|V(H)|}\in V(\mathcal{C}_k(G))$, and some tuple of vertices $\mathbf{x}=(v_1,\dots,v_d)\in (V(G))_d$, where $\mathbf{x}$ contains exactly the vertices that change color across the $\phi_i$. Let $\mathcal{P}_i=(c_{i,1},\dots,c_{i,r_i})$ denote the color palette of $v_i$, ordered according to the first coloring at which $v_i$ exhibits a given color. If a given valid state map $f$ of $H$ containing tuples of length $d$ accurately describes the index $j_i$ of the color of each $v_i$ within its color palette at each $\phi_k$, the colorings $\phi_1,\dots,\phi_{|V(H)|}\in V(\mathcal{C}_k(G))$ must form an induced $H\subseteq \mathcal{C}_k(G)$. This is a direct consequence of the condition that two vertices in $H$ have tuples differing at exactly one index if and only if they are adjacent. Whenever two tuples $(j_1,\dots,j_i,\dots,j_d)$ and $(j_1,\dots,j_i',\dots,j_d)$ differ at exactly one index, this indicates that the corresponding colorings $\phi,\phi'$ color $G$ such that $\phi(v_i)=c_{i,j_i}$ and $\phi'(v_i)=c_{i,j_i'}$, but all other vertices have matching colors across $\phi,\phi'$. This is precisely the definition of adjacent colorings of $G$.

Given a base graph $G$ and a valid state map $f\in\mathcal{F}(H,\mathcal{O})$ where $f\colon V(H)\to\mathbb{N}^{d}$, we construct a shadow graph $G_{f,\mathbf{x}}$ from $G$ on an ordered $d$-tuple $\mathbf{x}=(v_1, \dots, v_d)\in (V(G))_d$ as follows: 
\begin{itemize}
    \item For each $1\leq i\leq d$, replace each $v_i$ with a clique of size $r_i$ containing shadow vertices $v_{i,j} : 1 \leq j \leq r_i$ with $r_i$ being the maximal color index taken by $v_i$ across all states in $f$.
    \item Add edges between each clique and $N(v_i)\setminus \mathbf{x}$.
    \item For each pair of color palette indices in each state in $f(V(H))$, say $j_1$ in spot $i_1$ and $j_2$ in spot $i_2$, add an edge $v_{i_1,j_1}v_{i_2,j_2}$ if $v_{i_1}$ and $v_{i_2}$ are adjacent in $G$.
\end{itemize}
The edges between shadow cliques are added to ensure that the colorings dictated by states in $f$ actually result in proper colorings once the color palettes are chosen. For example, if some state $(2,5,\dots)$ is in the image of $f$, this dictates that for some $\phi$ in the resulting $H$-subgraph, we will have $\phi(v_1)=c_{1,3}$ and $\phi(v_2)=c_{2,5}$. If $v_1$ and $v_2$ are adjacent in $G$ we should ensure that $c_{1,3}\ne c_{2,5}$ by an edge between $v_{1,3}$ and $v_{2,5}$.

Given a coloring $c$ of $G_{f,\mathbf{x}}$, we can construct $|V(H)|$ colorings $\phi_1, \dots, \phi_{|V(H)|}$ of $G$ according to $f$. Let $\phi_k(w)=c(w)$ for all $w\notin \mathbf{x}$, and let $\phi_k(v_i)=c(v_{i,j_i})$ where $f(u_k)=(j_1,\dots,j_d)$. That is, the equality $f(u_k)=(j_1,\dots,j_d)$ dictates that vertex $v_i$ should take the color of the $j_i$th vertex in its shadow clique. Edges where one endpoint is not in $\mathbf{x}$ are proper in $G$ as they inherit their properness from their corresponding edge in $G_{f,\mathbf{x}}$. For edges $v_{i}v_{i'}\in E(G)$ where $v_{i},v_{i'}\in \mathbf{x}$, we know that in a coloring $\phi$ corresponding to some state $(j_1,\dots,j_d)$ in the image of $f$, we have added the edge $v_{i,j_i}v_{i',j_{i'}}$ to $E(G_{f,\mathbf{x}})$. Therefore these edges will not violate the properness of $\phi$. Conversely, given an induced $H\subseteq \mathcal{C}_k(G)$, and a traversal of $H$ according to $\mathcal{O}$, we can determine $f$ and $\mathbf{x}$ unambiguously by ordering the color changing vertices and their respective color palettes. We can then color each shadow clique in $G_{f,\mathbf{x}}$ of each vertex in $\mathbf{x}$ with its respective ordered color palette, and the properness of this coloring is ensured by the properness of all colorings in $V(H)$. To calculate $\pi_G^{(H)}$, we sum over $\pi_{G_{f,\mathbf{x}}}(k)$ for all $f\in\mathcal{F}(H,\mathcal{O})$, for all $\mathbf{x}\in (V(G))_d$ and divide by $|\operatorname{Aut}(H)|$. 

The final division by $|\operatorname{Aut}(H)|$ is necessary because each coloring counted by each $\pi_{G_{f,\mathbf{x}}}(k)$ corresponds to an ordering of some set of $|V(H)|$ colorings in $\mathcal{C}_k(G)$ satisfying a specific requirement on which pairs of colorings must have an edge between them, and which must not. If one such set of colorings can be relabeled in multiple ways to satisfy the same adjacency requirements, these different orderings will independently contribute to the sum. This may correspond to some permutation of the labels of the $v_i$ or the $c_{i,j}$. We can partition the set of all such ordered tuples of colorings in $\mathcal{C}_k(G)$ by which induced $H$-subgraph they span. Each partition will have exactly $|\operatorname{Aut}(H)|$ such orderings.

\subsection{The Construction for Disconnected Graphs}

Counting induced disconnected subgraphs $H\subseteq \mathcal{C}_k(G)$ can be done by multiplying the number of induced subgraphs isomorphic to each connected component, and subtracting out all possible ways of embedding the connected components in a way that is not induced. This involves counting all graphs containing all connected components of $H$, but which are not $H$, and do not have extra vertices. Formally, this can be defined as follows.

We define the notion of a graph containing two connected components without extra vertices. Given two graphs $A$ and $B$, let
$$
J(A, B) = \{\text{graphs } G  \ | \ 
 V(G)=V_A\cup V_B \text{ with } G[V_A]\cong A \text{ and } G[V_B]\cong B \}.
$$

For example when $A=K_3$ and $B=K_2$, the collection $J(A,B)$ consists of the following graphs:

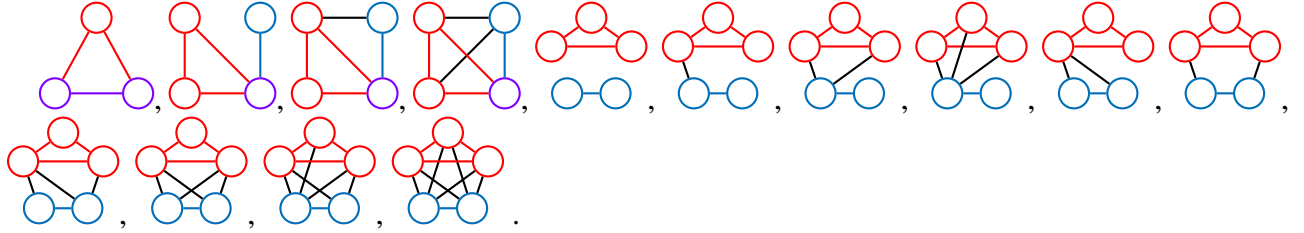
\begin{tikzpicture}
    \node[red_vertex] (v1) at (0.55, 1) {};
    \node[purple_vertex] (v2) at (1.1, 0) {};
    \node[purple_vertex] (v3) at (0, 0) {};
    \draw[red_edge] (v2) -- (v1) -- (v3);
    \draw[purple_edge] (v2) -- (v3);
\end{tikzpicture},
\begin{tikzpicture}
    \node[red_vertex] (v1) at (0, 0) {};
    \node[purple_vertex] (v2) at (1, 0) {};
    \node[red_vertex] (v3) at (0, 1) {};
    \node[blue_vertex] (v4) at (1, 1) {};
    \draw[red_edge] (v1) -- (v2) -- (v3) -- (v1);
    \draw[blue_edge] (v2) -- (v4);
\end{tikzpicture},
\begin{tikzpicture}
    \node[red_vertex] (v1) at (0, 0) {};
    \node[purple_vertex] (v2) at (1, 0) {};
    \node[red_vertex] (v3) at (0, 1) {};
    \node[blue_vertex] (v4) at (1, 1) {};
    \draw[red_edge] (v1) -- (v2) -- (v3) -- (v1);
    \draw[blue_edge] (v2) -- (v4);
    \draw[edge] (v3) -- (v4);
\end{tikzpicture},
\begin{tikzpicture}
    \node[red_vertex] (v1) at (0, 0) {};
    \node[purple_vertex] (v2) at (1, 0) {};
    \node[red_vertex] (v3) at (0, 1) {};
    \node[blue_vertex] (v4) at (1, 1) {};
    \draw[blue_edge] (v2) -- (v4);
    \draw[edge] (v3) -- (v4);
    \draw[edge] (v1) -- (v4);
    \draw[red_edge] (v1) -- (v2) -- (v3) -- (v1);
\end{tikzpicture},
\begin{tikzpicture}
    \node[red_vertex] (v1) at (0, 0.55) {};
    \node[red_vertex] (v2) at (0.53, 0.17) {};
    \node[blue_vertex] (v3) at (0.32, -0.45) {};
    \node[blue_vertex] (v4) at (-0.32, -0.45) {};
    \node[red_vertex] (v5) at (-0.53, 0.17) {};
    \draw[red_edge] (v1) -- (v2) -- (v5) -- (v1);
    \draw[blue_edge] (v3) -- (v4);
\end{tikzpicture},
\begin{tikzpicture}
    \node[red_vertex] (v1) at (0, 0.55) {};
    \node[red_vertex] (v2) at (0.53, 0.17) {};
    \node[blue_vertex] (v3) at (0.32, -0.45) {};
    \node[blue_vertex] (v4) at (-0.32, -0.45) {};
    \node[red_vertex] (v5) at (-0.53, 0.17) {};
    \draw[red_edge] (v1) -- (v2) -- (v5) -- (v1);
    \draw[blue_edge] (v3) -- (v4);
    \draw[edge] (v4) -- (v5);
\end{tikzpicture},
\begin{tikzpicture}
    \node[red_vertex] (v1) at (0, 0.55) {};
    \node[red_vertex] (v2) at (0.53, 0.17) {};
    \node[blue_vertex] (v3) at (0.32, -0.45) {};
    \node[blue_vertex] (v4) at (-0.32, -0.45) {};
    \node[red_vertex] (v5) at (-0.53, 0.17) {};
    \draw[red_edge] (v1) -- (v2) -- (v5) -- (v1);
    \draw[blue_edge] (v3) -- (v4);
    \draw[edge] (v2) -- (v4) -- (v5);
\end{tikzpicture},
\begin{tikzpicture}
    \node[red_vertex] (v1) at (0, 0.55) {};
    \node[red_vertex] (v2) at (0.53, 0.17) {};
    \node[blue_vertex] (v3) at (0.32, -0.45) {};
    \node[blue_vertex] (v4) at (-0.32, -0.45) {};
    \node[red_vertex] (v5) at (-0.53, 0.17) {};
    \draw[red_edge] (v1) -- (v2) -- (v5) -- (v1);
    \draw[blue_edge] (v3) -- (v4);
    \draw[edge] (v2) -- (v4) -- (v5);
    \draw[edge] (v1) -- (v4);
\end{tikzpicture},
\begin{tikzpicture}
    \node[red_vertex] (v1) at (0, 0.55) {};
    \node[red_vertex] (v2) at (0.53, 0.17) {};
    \node[blue_vertex] (v3) at (0.32, -0.45) {};
    \node[blue_vertex] (v4) at (-0.32, -0.45) {};
    \node[red_vertex] (v5) at (-0.53, 0.17) {};
    \draw[red_edge] (v1) -- (v2) -- (v5) -- (v1);
    \draw[blue_edge] (v3) -- (v4);
    \draw[edge] (v4) -- (v5) -- (v3);
\end{tikzpicture},
\begin{tikzpicture}
    \node[red_vertex] (v1) at (0, 0.55) {};
    \node[red_vertex] (v2) at (0.53, 0.17) {};
    \node[blue_vertex] (v3) at (0.32, -0.45) {};
    \node[blue_vertex] (v4) at (-0.32, -0.45) {};
    \node[red_vertex] (v5) at (-0.53, 0.17) {};
    \draw[red_edge] (v1) -- (v2) -- (v5) -- (v1);
    \draw[blue_edge] (v3) -- (v4);
    \draw[edge] (v4) -- (v5);
    \draw[edge] (v3) -- (v2);
\end{tikzpicture},
\begin{tikzpicture}
    \node[red_vertex] (v1) at (0, 0.55) {};
    \node[red_vertex] (v2) at (0.53, 0.17) {};
    \node[blue_vertex] (v3) at (0.32, -0.45) {};
    \node[blue_vertex] (v4) at (-0.32, -0.45) {};
    \node[red_vertex] (v5) at (-0.53, 0.17) {};
    \draw[red_edge] (v1) -- (v2) -- (v5) -- (v1);
    \draw[blue_edge] (v3) -- (v4);
    \draw[edge] (v4) -- (v5) -- (v3) -- (v2);
\end{tikzpicture},
\begin{tikzpicture}
    \node[red_vertex] (v1) at (0, 0.55) {};
    \node[red_vertex] (v2) at (0.53, 0.17) {};
    \node[blue_vertex] (v3) at (0.32, -0.45) {};
    \node[blue_vertex] (v4) at (-0.32, -0.45) {};
    \node[red_vertex] (v5) at (-0.53, 0.17) {};
    \draw[red_edge] (v1) -- (v2) -- (v5) -- (v1);
    \draw[blue_edge] (v3) -- (v4);
    \draw[edge] (v4) -- (v5) -- (v3) -- (v2);
    \draw[edge] (v4) -- (v2);
\end{tikzpicture},
\begin{tikzpicture}
    \node[red_vertex] (v1) at (0, 0.55) {};
    \node[red_vertex] (v2) at (0.53, 0.17) {};
    \node[blue_vertex] (v3) at (0.32, -0.45) {};
    \node[blue_vertex] (v4) at (-0.32, -0.45) {};
    \node[red_vertex] (v5) at (-0.53, 0.17) {};
    \draw[blue_edge] (v3) -- (v4);
    \draw[edge] (v4) -- (v5) -- (v3) -- (v2);
    \draw[edge] (v2) -- (v4) -- (v1);
    \draw[red_edge] (v1) -- (v2) -- (v5) -- (v1);
\end{tikzpicture},
\begin{tikzpicture}
    \node[red_vertex] (v1) at (0, 0.55) {};
    \node[red_vertex] (v2) at (0.53, 0.17) {};
    \node[blue_vertex] (v3) at (0.32, -0.45) {};
    \node[blue_vertex] (v4) at (-0.32, -0.45) {};
    \node[red_vertex] (v5) at (-0.53, 0.17) {};
    \draw[blue_edge] (v3) -- (v4);
    \draw[edge] (v4) -- (v5) -- (v3) -- (v2);
    \draw[edge] (v2) -- (v4) -- (v1) -- (v3);
    \draw[red_edge] (v1) -- (v2) -- (v5) -- (v1);
\end{tikzpicture}
.

Given a graphs $G$, $A$, and $B$, define $f_G(A,B)$ to count the number of embeddings of $A,B$ into $G$ without leftover vertices as follows:
$$
f_G(A, B) = \# \{(V_A, V_B) \ | \  V(G)=V_A\cup V_B \text{ with } G[V_A]\cong A \text{ and } G[V_B]\cong B \}.
$$
Note that $V_A,V_B$ may overlap. For example $f_{\text{Diamond}}(K_2,K_3)=4$, counting the following 4 embeddings:
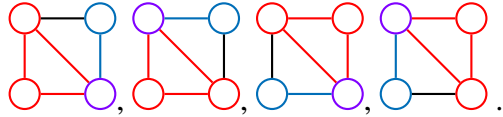
\begin{center}
    \begin{tikzpicture}
    \node[red_vertex] (v1) at (0, 0) {};
    \node[purple_vertex] (v2) at (1, 0) {};
    \node[red_vertex] (v3) at (0, 1) {};
    \node[blue_vertex] (v4) at (1, 1) {};
    \draw[red_edge] (v1) -- (v2) -- (v3) -- (v1);
    \draw[blue_edge] (v2) -- (v4);
    \draw[edge] (v3) -- (v4);
    \end{tikzpicture},
    \begin{tikzpicture}
    \node[red_vertex] (v1) at (0, 0) {};
    \node[red_vertex] (v2) at (1, 0) {};
    \node[purple_vertex] (v3) at (0, 1) {};
    \node[blue_vertex] (v4) at (1, 1) {};
    \draw[red_edge] (v1) -- (v2) -- (v3) -- (v1);
    \draw[edge] (v2) -- (v4);
    \draw[blue_edge] (v3) -- (v4);
    \end{tikzpicture},
    \begin{tikzpicture}
    \node[blue_vertex] (v1) at (0, 0) {};
    \node[purple_vertex] (v2) at (1, 0) {};
    \node[red_vertex] (v3) at (0, 1) {};
    \node[red_vertex] (v4) at (1, 1) {};
    \draw[red_edge] (v4) -- (v2) -- (v3) -- (v4);
    \draw[blue_edge] (v2) -- (v1);
    \draw[edge] (v3) -- (v1);
    \end{tikzpicture},
    \begin{tikzpicture}
    \node[blue_vertex] (v1) at (0, 0) {};
    \node[red_vertex] (v2) at (1, 0) {};
    \node[purple_vertex] (v3) at (0, 1) {};
    \node[red_vertex] (v4) at (1, 1) {};
    \draw[red_edge] (v4) -- (v2) -- (v3) -- (v4);
    \draw[edge] (v2) -- (v1);
    \draw[blue_edge] (v3) -- (v1);
    \end{tikzpicture}
    .\end{center}

For a disconnected graph $H$, we calculate $\pi^{(H)}_G(k)$ by recursion on the connected components. Suppose we select one such component, $H'$. Given $H'$ and the complement $H\setminus H'$ (which may still be disconnected), 
then $$\pi^{(H)}_G(k)=\frac{1}{n_{H'}(H)}\left[\pi^{(H')}_G(k)\cdot \pi^{(H\setminus H')}_G(k)-\sum_{U\in J(H', H\setminus H')\setminus \{H\}} f_U(H', H\setminus H') \cdot \pi^{(U)}_{G}(k)\right]$$
Where $n_{H'}(H)$ counts the number of connected components of $H$ isomorphic to $H'$. 

\begin{proof}
    The product $\pi^{(H')}_G(k)\cdot \pi^{(H\setminus H')}_G(k)$ counts all pairs of induced subgraphs $(H'_i,(H\setminus H')_j)$ where $H'_i,(H\setminus H')_j\subseteq\mathcal{C}_k(G)$. Some of these pairs do not constitute an induced subgraph isomorphic to $H$, as they may be overlapping or have extra edges in $\mathcal{C}_k(G)$ connecting them. If this is the case, the true induced graph on their vertex set must be some other graph that contains both $H'$ and $H\setminus H'$ as induced subgraphs and has no additional vertices. This is precisely the definition of $J(H',H\setminus H')$. For each such graph, we should subtract out how many invalid pairs of $(H'_i,(H\setminus H')_j)$ in $\mathcal{C}_k(G)$ actually induce this graph. This is the definition of $f_U(H', H\setminus H')$. By subtracting all such invalid pairs, we are left with only the pairs that give an induced subgraph isomorphic to $H$. However, since $H$ contains $n_{H'}(H)$ connected components isomorphic to $H'$, then for any induced subgraph of $\mathcal{C}_k(G)$ isomorphic to $H$, there will be $n_{H'}(H)$ pairs of $(H'_i,(H\setminus H')_j)$ which are valid. These can be identified by selecting each $H'$-component of the $H$-subgraph as the prior element of the pair, and the remainder of the $H$-subgraph as the latter. To account for this we divide by $n_{H'}(H)$.
\end{proof}

\section{Hypercube Polynomials of Trees}\label{sec:hypercube}

In this section we will demonstrate how the shadow graphs established in the previous section can be used to give a closed form for polynomials counting induced hypercubes in trees, a task which did not seem to be attainable using previous methods. As hypercubes are a special, highly symmetric case of a Cartesian product of cliques, we can use the formula from Subsection~\ref{subsec:cartesian}. For simplicity, because we are calculating the special case of a Cartesian product of cliques where all cliques are the same size, we can simply sum over unordered sets of $d$ vertices, and omit the factor of $\frac{1}{d!}$. It is known that the chromatic polynomial of a tree depends only on the number of vertices in the tree, and that the chromatic pairs polynomial depends only on the degree sequence. Therefore it is natural to wonder what data is necessary to construct higher dimensional hypercube polynomials. We find a simple expression for the square-counting polynomial and the general $d$-dimensional hypercube polynomial. 

\subsection{Chromatic and Chromatic Pairs Polynomial of a Tree}

For a tree $T$ with $n$ vertices, the chromatic polynomial is well-known to be $\pi_T(k) = k(k - 1)^{n - 1}$. This reflects that after choosing a color for a root vertex, each subsequent vertex has $k - 1$ choices to avoid matching its parent. The chromatic polynomial is the same as $\pi_{T}^{(Q_0)}(k)$. To calculate the chromatic pairs polynomial $\pi_{T}^{(Q_1)}(k)=\pi_{T}^{(P_2)}(k)$, the following formula appears in \cite{AKLR25}*{Equation (4.4)}. We give an alternative derivation using shadow graphs.

\begin{prop}
For a tree $T$, the chromatic pairs polynomial is
$$
\pi_T^{(Q_1)}(k) = \frac{1}{2} \sum_{v \in V} k(k - 1)^{n - \deg(v)} (k - 2)^{\deg(v)}.
$$ 
\end{prop}

\begin{proof}
In the shadow graph $T_v$, where $v$ is replaced by $v_1$ and $v_2$, we have:
\begin{itemize}
    \item $k$ choices for $v_1$.
    \item $k - 1$ choices for $v_2$ (since $v_2$ must differ from $v_1$).
    \item $k - 2$ choices for all neighbors of $v_1$ (must differ from both $v_1$ and $v_2$).
    \item $k - 1$ choices for all other vertices (must differ from their parent).
\end{itemize}
An example of $T_v$ for one such $T,v$ is shown below, labeled according to the number of colors available:
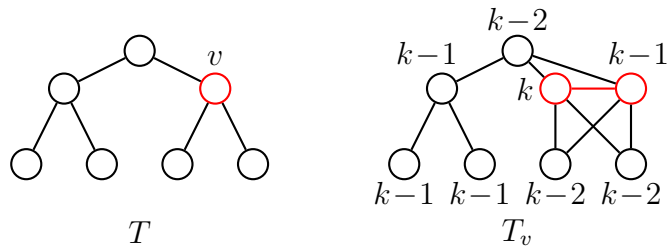
\begin{figure}[ht]
    \centering
    \begin{tikzpicture}
        \node[vertex] (0) at (-1, -0.5) {};
        \node[vertex] (1) at (-2, -0.5) {};
        \node[vertex] (2) at (-3, -0.5) {};
        \node[vertex] (3) at (-4, -0.5) {};
        \node[vertex] (4) at (-3.5, 0.5) {};
        \node[red_vertex] (5) at (-1.5, 0.5) {};
        \node[] (21) at (-1.5, 0.9) {$v$};
        \node[vertex] (6) at (-2.5, 1) {};
        \node[vertex] (7) at (1, -0.5) {};
        \node[] (21) at (1, -0.9) {$k\!-\!1$};
        \node[vertex] (8) at (2, -0.5) {};
        \node[] (21) at (2, -0.9) {$k\!-\!1$};
        \node[vertex] (9) at (3, -0.5) {};
        \node[] (21) at (3, -0.9) {$k\!-\!2$};
        \node[vertex] (10) at (4, -0.5) {};
        \node[] (21) at (4, -0.9) {$k\!-\!2$};
        \node[red_vertex] (11) at (3, 0.5) {};
        \node[] (21) at (2.6, 0.5) {$k$};
        \node[red_vertex] (12) at (4, 0.5) {};
        \node[] (21) at (4.1, 0.95) {$k\!-\!1$};
        \node[vertex] (13) at (2.5, 1) {};
        \node[] (21) at (2.5, 1.4) {$k\!-\!2$};
        \node[vertex] (14) at (1.5, 0.5) {};
        \node[] (21) at (1.3, 0.95) {$k\!-\!1$};
        \draw[edge] (4) -- (2);
        \draw[edge] (4) -- (3);
        \draw[edge] (5) -- (0);
        \draw[edge] (5) -- (1);
        \draw[edge] (6) -- (4);
        \draw[edge] (6) -- (5);
        \draw[edge] (11) -- (9);
        \draw[edge] (11) -- (10);
        \draw[edge] (12) -- (9);
        \draw[edge] (12) -- (10);
        \draw[red_edge] (12) -- (11);
        \draw[edge] (13) -- (11);
        \draw[edge] (13) -- (12);
        \draw[edge] (14) -- (7);
        \draw[edge] (14) -- (8);
        \draw[edge] (14) -- (13);
		\node[] (40) at (-2.5, -1.4) {$T$};
		\node[] (41) at (2.5, -1.4) {$T_v$};
        \end{tikzpicture}
\end{figure}

Multiplying these choices gives
$$
\pi_{T_v}(k) = k (k - 1) (k - 2)^{\deg(v)} (k - 1)^{n - \deg(v) - 1}.
$$
Simplifying, we get $\pi_{T_v}(k) = k (k - 1)^{n - \deg(v)} (k - 2)^{\deg(v)}$. Summing over all $v \in V(T)$ and dividing by $2$ (as per Lemma~\ref{lem:pairs}), we obtain the desired formula.
\end{proof}

\subsection{Square Polynomial of a Tree}

When coloring a shadow graph of a tree with two shadow vertices (say $u_1,u_2$ and $v_1,v_2$ replacing base vertices $u$ and $v$), the construction depends on whether $uv\in E(T)$. Below is an example of each scenario for one such $T$, in one case with $u,v_1$ and another with $u,v_2$. Both are labeled according to the number of colors available as we traverse the shadow graph:
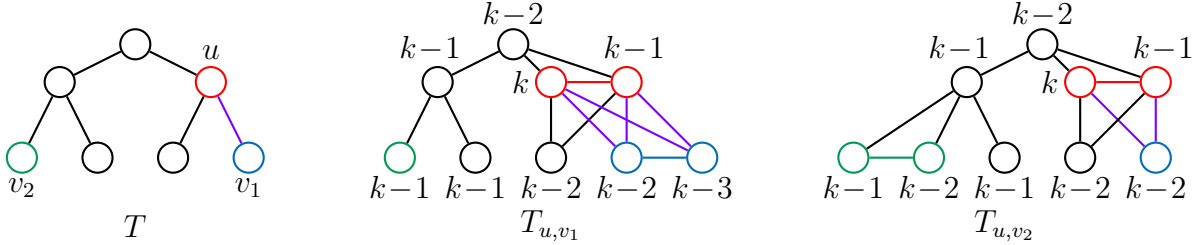
\begin{figure}[ht]
    \centering
    \begin{tikzpicture}
        \node[vertex] (0) at (0, 1) {};
        \node[] (40) at (0,1.4) {$k\!-\!2$};
        \node[red_vertex] (1) at (0.5, 0.5) {};
        \node[] (40) at (0.1, 0.5) {$k$};
        \node[red_vertex] (2) at (1.5, 0.5) {};
        \node[] (40) at (1.6, 0.95) {$k\!-\!1$};
        \node[vertex] (3) at (0.5, -0.5) {};
        \node[] (40) at (0.5,-0.9) {$k\!-\!2$};
        \node[blue_vertex] (4) at (1.5, -0.5) {};
        \node[] (40) at (1.5,-0.9) {$k\!-\!2$};
        \node[blue_vertex] (5) at (2.5, -0.5) {};
        \node[] (40) at (2.5,-0.9) {$k\!-\!3$};
        \node[vertex] (6) at (-1, 0.5) {};
        \node[] (40) at (-1.1,0.95) {$k\!-\!1$};
        \node[vertex] (7) at (-0.5, -0.5) {};
        \node[] (40) at (-0.5,-0.9) {$k\!-\!1$};
        \node[green_vertex] (8) at (-1.5, -0.5) {};
        \node[] (40) at (-1.5,-0.9) {$k\!-\!1$};
        \node[blue_vertex] (9) at (-3.5, -0.5) {};
        \node[] (40) at (-3.5,-0.9) {$v_1$};
        \node[vertex] (10) at (-4.5, -0.5) {};
        \node[vertex] (11) at (-5.5, -0.5) {};
        \node[green_vertex] (12) at (-6.5, -0.5) {};
        \node[] (40) at (-6.5,-0.9) {$v_2$};
        \node[vertex] (13) at (-6, 0.5) {};
        \node[red_vertex] (14) at (-4, 0.5) {};
        \node[] (40) at (-4,0.9) {$u$};
        \node[vertex] (15) at (-5, 1) {};
        \node[green_vertex] (17) at (4.5, -0.5) {};
        \node[] (40) at (4.5,-0.9) {$k\!-\!1$};
        \node[green_vertex] (18) at (5.5, -0.5) {};
        \node[] (40) at (5.5,-0.9) {$k\!-\!2$};
        \node[vertex] (19) at (6.5, -0.5) {};
        \node[] (40) at (6.5,-0.9) {$k\!-\!1$};
        \node[vertex] (20) at (7.5, -0.5) {};
        \node[] (40) at (7.5,-0.9) {$k\!-\!2$};
        \node[blue_vertex] (21) at (8.5, -0.5) {};
        \node[] (40) at (8.5,-0.9) {$k\!-\!2$};
        \node[vertex] (22) at (6, 0.5) {};
        \node[] (40) at (5.9,0.95) {$k\!-\!1$};
        \node[red_vertex] (23) at (7.5, 0.5) {};
        \node[] (40) at (7.1,0.5) {$k$};
        \node[red_vertex] (24) at (8.5, 0.5) {};
        \node[] (40) at (8.6,0.95) {$k\!-\!1$};
        \node[vertex] (25) at (7, 1) {};
        \node[] (40) at (7,1.4) {$k\!-\!2$};
        \draw[edge] (1) -- (0);
        \draw[edge] (2) -- (0);
        \draw[red_edge] (2) -- (1);
        \draw[edge] (3) -- (1);
        \draw[edge] (3) -- (2);
        \draw[purple_edge] (4) -- (1);
        \draw[purple_edge] (4) -- (2);
        \draw[purple_edge] (5) -- (1);
        \draw[purple_edge] (5) -- (2);
        \draw[blue_edge] (5) -- (4);
        \draw[edge] (6) -- (0);
        \draw[edge] (7) -- (6);
        \draw[edge] (8) -- (6);
        \draw[edge] (13) -- (11);
        \draw[edge] (13) -- (12);
        \draw[purple_edge] (14) -- (9);
        \draw[edge] (14) -- (10);
        \draw[edge] (15) -- (13);
        \draw[edge] (15) -- (14);
        \draw[green_edge] (18) -- (17);
        \draw[edge] (22) -- (17);
        \draw[edge] (22) -- (18);
        \draw[edge] (22) -- (19);
        \draw[edge] (23) -- (20);
        \draw[purple_edge] (23) -- (21);
        \draw[edge] (24) -- (20);
        \draw[purple_edge] (24) -- (21);
        \draw[red_edge] (24) -- (23);
        \draw[edge] (25) -- (22);
        \draw[edge] (25) -- (23);
        \draw[edge] (25) -- (24);
		\node[] (40) at (-5, -1.4) {$T$};
		\node[] (41) at (0.5, -1.4) {$T_{u,v_1}$};
		\node[] (41) at (6.5, -1.4) {$T_{u,v_2}$};
    \end{tikzpicture}
\end{figure}
\smallskip 

\noindent\textit{Case 1: $u$ and $v$ are adjacent.}

If $u$ and $v$ are adjacent in $T$, we show that the contribution to the $2$-dimensional hypercube polynomial is
$$
k (k - 1)^{n - \deg(u) - \deg(v) + 1} (k - 2)^{\deg(u) + \deg(v) - 1} (k - 3).
$$
We can color the vertices of the shadow graph in the following order:
\begin{itemize}
    \item $k$ choices for $u_1$.
    \item $k - 1$ choices for $u_2$, as it must differ from $u_1$.
    \item $k - 2$ choices for $v_1$, as it must differ from $u_1$ and $u_2$.
    \item $k - 3$ choices for $v_2$, as it must differ from $u_1$, $u_2$, and $v_1$.
    \item All neighbors of $u_1$ and $v_1$ (excluding each other) have $k - 2$ choices as neighbors of $u_1$ must differ from $u_1$ and $u_2$, and likewise for neighbors of $v_1$. This gives a factor of $(k-2)^{\deg(u)+\deg(v)-2}$.
    \item Remaining vertices have can be assigned colors by traversing the remainder of the tree, with $k - 1$ choices each. This gives a factor of $(k-1)^{n-\deg(u)-\deg(v)}$.
\end{itemize}
Multiplying these choices gives the formula.

\smallskip 

\noindent\textit{Case 2: $u$ and $v$ are \textbf{not} adjacent.}

If $u$ and $v$ are not adjacent in $T$, we show the corresponding contribution is
$$
k (k - 1)^{n - \deg(u) - \deg(v) + 1} (k - 2)^{\deg(u) + \deg(v)}.
$$
The coloring choices are:
\begin{itemize}
    \item $k$ choices for $u_1$.
    \item $k - 1$ choices for $u_2$ as it must differ from $u_1$.
    \item $k - 2$ choices for all neighbors of $u_1$ as they must differ from $u_1$ and $u_2$.
    \item $k - 1$ choices for all nodes on the path between $u_1$ and $v_1$ excluding $u_1$, and the neighbor of $u_1$, but including $v_1$, as each of these can be assigned a color differing only from their one neighbor which has already been assigned a color.
    \item $k - 2$ choices for $v_2$ as it must differ from $v_1$ and the neighbor of $v_1$ on the path between $u_1$ and $v_1$.
    \item $k - 2$ choices for all other neighbors of $v_1$, as they must differ from $v_1$ and $v_2$.
    \item $k - 1$ choices for all other vertices, as they can be assigned colors according to an ordered traversal, in which they only need differ from their neighbor which has already been assigned a color.
\end{itemize}
The length of the path between $u$ and $v$ does not affect the total count, as each vertex on the path has an identical number of choices to all other vertices of distance at least 2 to $u$ and $v$. Multiplying the terms yields the formula.

Our analysis above leads to the following result:
\begin{cor}
\begin{align*}
    \pi^{(Q_2)}_T(k)&= \frac{1}{4}\sum_{uv\in E(T)} k (k - 1)^{n - \deg(u) - \deg(v) + 1} (k - 2)^{\deg(u) + \deg(v) - 1} (k - 3) \\ 
    & + \frac{1}{4}\sum_{uv\notin E(T)}k (k - 1)^{n - \deg(u) - \deg(v) + 1} (k - 2)^{\deg(u) + \deg(v)}.
\end{align*}
\end{cor}

\subsection{General Case for \texorpdfstring{$d$}{d} Shadow Vertices}

In the general case, when coloring the tree with $d$ shadow vertices, we sum over all unordered sets $S \subseteq V$ with $|S| = d$.

\begin{theorem}\label{thm:hypercube-trees}
The $d$-dimensional hypercube polynomial for a tree $T$ is
$$
\pi_T^{(Q_d)}(k) =\frac{1}{2^d} \sum_{S \subseteq V(T), |S| = d} k (k - 1)^{n + d - 1 - \deg(S)} (k - 2)^{\deg(S) - \Int(S)} (k - 3)^{\Int(S)},
$$
where $\deg(S) = \sum_{v \in S} \deg(v)$ and $\Int(S)$ is the number of interior edges, those with two endpoints in $S$.
\end{theorem}

\begin{proof}
For each set $S=\{v_1,v_2,\dots,v_d\}$ of vertices in $T$, we:
\begin{itemize}
    \item Replace each vertex $v_i$ in $S$ with two shadow vertices $v_{i,1}$ and $v_{i,2}$.
    \item Add edges $v_{i,1}v_{i,2}$ for all $i\le d$, as well as all edges in $\{v_{i,1},v_{i,2}\}\times N(v_i)$.
    \item Add edges $\{v_{i,1},v_{i,2}\}\times \{v_{j,1},v_{j,2}\}$ for all $v_iv_j\in E(T)$.
\end{itemize}

To color this graph, we begin by assigning any of the $k$ colors to $v_{1,1}$, and any of the $k-1$ remaining colors to $v_{1,2}$. Then, traversing the tree, we make the following coloring choices:
\begin{itemize}
    \item $k-2$ choices for any vertex not in $S$ adjacent to a shadow $K_2$ that has already been colored.
    \item $(k-1)(k-2)$ choices for both vertices in a shadow $K_2$ which is not adjacent to another shadow $K_2$ that has already been colored.
    \item $(k-2)(k-3)$ choices for both vertices in a shadow $K_2$ which is adjacent to another shadow $K_2$ that has already been colored.
    \item $k-1$ choices for a vertex not in $S$ which is not adjacent to a shadow $K_2$ that has already been colored.
\end{itemize}
In total, there will be a single factor of $k$. There will be one factor of $k-3$ for each shadow $K_2$, minus one for each connected component, as the first pair to be assigned a color in a connected component will have $(k-1)(k-2)$ choices. This is equivalent to the number of edges in $T$ with two endpoints in $S$. There will be a factor of $k-2$ for all vertices not in $S$ that are adjacent to a vertex in $S$, except for the first vertex to receive a color. There will be an additional factor of $k-2$ for all pairs of shadow vertices. This is equivalent to the number of edges in $T$ with at least one endpoint in $S$, which is $\deg(S)-\Int(S)$. The remaining vertices contribute a factor of $k-1$. Dividing by $2^d$ corrects for the over-counting due to permuting the colors in the color palette of each shadow vertex.
\end{proof}

The data necessary to construct the $d$th hypercube polynomial is then the multiset of all pairs $(\deg(S),\Int(S))$ for all subsets $S\subseteq V(T)$ where $|S|=d$. The number of such subsets, $\binom{|V(T)|}{d}$ increases until $d=\lfloor\frac{|V(T)|}{2}\rfloor$, at which point $\binom{|V(T)|}{d}=\binom{|V(T)|}{|V(T)|-d}$ for $d>\lfloor\frac{|V(T)|}{2}\rfloor$. In fact, for any subset $S$ with degree $\deg(S)$, the complement of $S$, $\overline{S}=V(T)\setminus S$ must have degree $\deg(\overline{S})=2\left(|V(T)|-1\right)-\deg(S)$, as $\deg(S)+\deg(\overline{S})=2|E(T)|$ by the handshake lemma, and $|E(T)|=|V(T)|-1$. Likewise, $\Int(\overline{S})=|V(T)|-1-\deg(S)-\Int(S)$, as $|E(T)|=\Int(S)+(\deg(S)-2\Int(S))+\Int(\overline{S})$. Therefore the $d$th multiset can be used to construct the $(|V(T)|-d)$th multiset.

\section{Relation to Generalized Degree Sequence}\label{sec:gen-deg-seq}
Having confirmed that $d$-dimensional hypercube polynomials do indeed require more data to calculate up until $d=\lfloor\frac{|V(T)|}{2}\rfloor$, it is natural to wonder whether the full set of hypercube polynomials form a complete invariant for trees. In this section, by relating this data to the generalized degree sequence introduced in \cite{Cre20}, we answer this question in the negative. 

The generalized degree sequence is defined as the multiset of triplets $(|S|, \Int(S), \Ext(S))$, for all subsets of $S\subseteq V(T)$, where $\Ext(S) = \deg(S) - 2 \Int(S)$ is the number of edges with exactly one endpoint in $S$. Since we can extract $\Int(S),\deg(S)$ from the generalized degree sequence, the generalized degree sequence is sufficient to determine the full set of hypercube polynomials of a tree by Theorem~\ref{thm:hypercube-trees}.

However, as shown in \cite{Cre20}, the generalized degree sequence is not a complete invariant of trees; non-isomorphic trees can share the same set of hypercube polynomials. The minimal example of this is the following pair of trees:

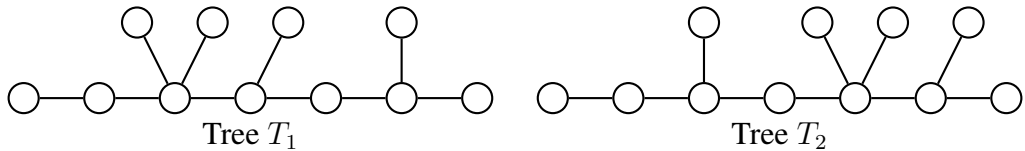
\begin{figure}[ht]
    \centering
    \begin{tikzpicture}
        \node[vertex] (S1) at (0,0) {};
        \node[vertex] (S2) at (1,0) {};
        \node[vertex] (S3) at (2,0) {};
        \node[vertex] (S3L) at (1.5,1) {};
        \node[vertex] (S3R) at (2.5,1) {};
        \node[vertex] (S4) at (3,0) {};
        \node[vertex] (S4U) at (3.5,1) {};
        \node[vertex] (S5) at (4,0) {};
        \node[vertex] (S6) at (5,0) {};
        \node[vertex] (S6U) at (5,1) {};
        \node[vertex] (S7) at (6,0) {};
        \draw[edge] (S1) -- (S2) -- (S3) -- (S4) -- (S5) -- (S6) -- (S7);
        \draw[edge] (S3R) -- (S3) -- (S3L);
        \draw[edge] (S4U) -- (S4);
        \draw[edge] (S6U) -- (S6);
        \node at (3, -0.5) {Tree $T_1$};

        \node[vertex] (T1) at (7,0) {};
        \node[vertex] (T2) at (8,0) {};
        \node[vertex] (T3) at (9,0) {};
        \node[vertex] (T3U) at (9,1) {};
        \node[vertex] (T4) at (10,0) {};
        \node[vertex] (T5) at (11,0) {};
        \node[vertex] (T5L) at (10.5,1) {};
        \node[vertex] (T5R) at (11.5,1) {};
        \node[vertex] (T6) at (12,0) {};
        \node[vertex] (T6U) at (12.5,1) {};
        \node[vertex] (T7) at (13,0) {};
        \draw[edge] (T1) -- (T2) -- (T3) -- (T4) -- (T5) -- (T6) -- (T7);
        \draw[edge] (T5R) -- (T5) -- (T5L);
        \draw[edge] (T3U) -- (T3);
        \draw[edge] (T6U) -- (T6);
        \node at (10, -0.5) {Tree $T_2$};
    \end{tikzpicture}
\end{figure}

Letting $M_i(T)$ denote the multiset of pairs $(\Int(X),\Ext(X))$ across all subsets $X\subseteq V(T)$ where $|X|=i$, we have the following implications between hypercube polynomials and sections of the generalized degree sequence. An arrow $A(T)\implies B(T)$ indicates that for two trees $T_1,T_2$ where $A(T_1)=A(T_2)$ it is always the case that $B(T_1)=B(T_2)$. An arrow $A(T)\overset{?}{\implies} B(T)$ indicates that $A(T)\implies B(T)$ has been confirmed to hold for trees $T$ where $|V(T)| \le 20$. In the diagram below, for all implications, we consider only when $i\le \lfloor\frac{|V(T)|}{2}\rfloor$, as we have established the amount of data in $M_i(T)$ decreases past that point:
$$
\begin{array}{*{9}{c}}
& & M_1(T) & \impliedby & M_2(T) & \overset{?}{\impliedby} & M_3(T) & \overset{?}{\impliedby} & M_4(T) \\
& &\Updownarrow && \Downarrow\;\Uparrow ? && \Downarrow&& \Downarrow\\
\pi_T(k) & \impliedby & \pi_T^{(Q_1)}(k) & \overset{?}{\impliedby} &  \pi_T^{(Q_2)}(k)& \overset{?}{\impliedby} & \pi_T^{(Q_3)}(k)& \centernot{\impliedby} & \pi_T^{(Q_4)}(k)
\end{array}
$$

As mentioned above, the vertical arrows $M_j(T)\implies \pi_T^{(Q_j)}(k)$ are justified by Theorem~\ref{thm:hypercube-trees}. When $j=1$, the reverse implication $\pi_T^{(Q_1)}(k)\implies M_1(T)$ amounts to the statement that degree sequence of a tree can be recovered from the chromatic pairs polynomial \cite{AKLR25}*{Theorem 5.1}. Since $\pi_T(k)=k(k-1)^{|V(T)|-1}$, the implication $\pi_T^{(Q_1)}(k)\implies \pi_T(k)$ is immediate. Next, we verify the only remaining implication from the diagram.

\begin{prop}
We have $M_2(T) \implies M_1(T)$.
\end{prop}

\begin{proof}
The given multiset $M_2(T) = \{(\Int(X), \Ext(X))\}$ uniquely determines the multiset $D(T) = \{\deg(u) + \deg(v) \mid \{u,v\} \in \binom{V(T)}{2}\}$, since $\deg(u)+\deg(v) = \Ext(\{u,v\}) + 2\Int(\{u,v\})$. We show $D(T)$ determines the degree sequence $M_1(T) = d(T)$.

Let $\#(i,S)$ denote the number of occurrences of the element $i$ in the multiset $S$. The size $|D(T)| = \binom{n}{2}$ determines $n=|V(T)|$. Assume $n \ge 2$. The count $\#(2, D(T)) = \binom{\#(1,d(T))}{2}$, as any pair of vertices in $T$ whose degrees add to 2 must both be leaves. This allows us to determines the number of leaves $\#(1,d(T))$.

We proceed by strong induction on $k \ge 2$ to determine $\#(k,d(T))$, assuming $\#(i,d(T))$ are determined for all $1 \le i<k$. The count $\#(k+1, D(T))$ can be expressed as:
$$ \#(k+1, D(T)) = \sum_{\substack{i+j=k+1 \\ 1 \le i \le j}} \begin{cases} \#(i, d(T))\#(j, d(T)) & \text{if } i < j \\ \binom{\#(i, d(T))}{2} & \text{if } i = j \end{cases} $$
The term in this sum corresponding to $(i,j) = (1,k)$ is $\#(1, d(T))\#(k, d(T))$. All other terms correspond to pairs $(i,j)$ with $1 < i \le j$. For these terms, $i\le j \le k-1$. Thus, all terms in the sum, except for $\#(1, d(T))\#(k, d(T))$, depend only on counts $\#(i', d(T))$ where $i' < k$. By the induction hypothesis, these counts are known.
Therefore, $\#(1, d(T))\#(k, d(T))$ is known, and since $\#(1, d(T))$ is known, we can uniquely determine the value of $\#(k,d(T))$.

Thus, $D(T)$ determines $d(T) = M_1(T)$. Since $M_2(T)$ determines $D(T)$, the result follows.
\end{proof}

To prove that trees with identical 4-dimensional hypercube polynomials need not have identical 3-dimensional cube polynomials, we present the following pair of trees for which this is the case:
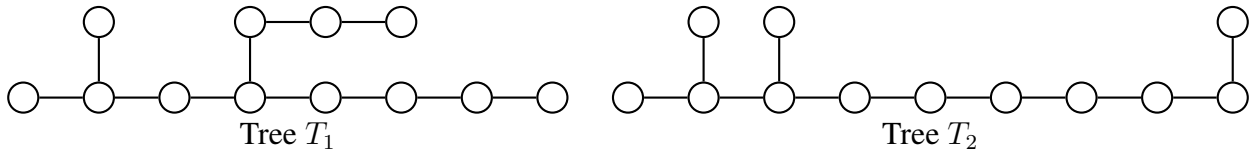
\begin{figure}[ht]
    \centering
    \begin{tikzpicture}
        \node[vertex] (S1) at (0,0) {};
        \node[vertex] (S2) at (1,0) {};
        \node[vertex] (S2U) at (1,1) {};
        \node[vertex] (S3) at (2,0) {};
        \node[vertex] (S4) at (3,0) {};
        \node[vertex] (S5) at (4,0) {};
        \node[vertex] (S6) at (5,0) {};
        \node[vertex] (S7) at (6,0) {};
        \node[vertex] (S8) at (7,0) {};
        \node[vertex] (S4U1) at (3,1) {};
        \node[vertex] (S4U2) at (4,1) {};
        \node[vertex] (S4U3) at (5,1) {};
        \draw[edge] (S1) -- (S2) -- (S3) -- (S4) -- (S5) -- (S6) -- (S7) -- (S8);
        \draw[edge] (S2U) -- (S2);
        \draw[edge] (S4) -- (S4U1) -- (S4U2) -- (S4U3);
        \node at (3.5, -0.5) {Tree $T_1$};

        \node[vertex] (T1) at (8,0) {};
        \node[vertex] (T2) at (9,0) {};
        \node[vertex] (T2U) at (9,1) {};
        \node[vertex] (T3) at (10,0) {};
        \node[vertex] (T3U) at (10,1) {};
        \node[vertex] (T4) at (11,0) {};
        \node[vertex] (T5) at (12,0) {};
        \node[vertex] (T6) at (13,0) {};
        \node[vertex] (T7) at (14,0) {};
        \node[vertex] (T8) at (15,0) {};
        \node[vertex] (T9) at (16,0) {};
        \node[vertex] (T10) at (16,1) {};
        \draw[edge] (T1) -- (T2) -- (T3) -- (T4) -- (T5) -- (T6) -- (T7) -- (T8) -- (T9) -- (T10);
        \draw[edge] (T2U) -- (T2);
        \draw[edge] (T3U) -- (T3);
        \node at (12, -0.5) {Tree $T_2$};
    \end{tikzpicture}
\end{figure}
\begin{align*}
\pi_{T_1}^{(Q_4)}(k) = \pi_{T_2}^{(Q_4)}(k) 
&= \frac{495}{16}k^{16} - \frac{5775}{8}k^{15} + \frac{62525}{8}k^{14} - 52077k^{13} \\
& + \frac{3820713}{16}k^{12} - \frac{6384613}{8}k^{11} + \frac{32128997}{16}k^{10} - \frac{15491779}{4}k^9 \\
& + \frac{577231}{1}k^8 - \frac{26590137}{4}k^7 + \frac{93863265}{16}k^6 - \frac{7791539}{2}k^5 \\
& + \frac{3768707}{2}k^4 - 626702k^3 + 128145k^2 - 12144k
\end{align*}
\begin{align*}
\pi_{T_1}^{(Q_3)}(k) &=\frac{55}{2}k^{15} - 550k^{14} + \frac{10147}{2}k^{13} - 28609k^{12} + \frac{220335}{2}k^{11} - \frac{2451663}{8}k^{10} \\
&+ 635055k^9 - \frac{7967177}{8}k^8 + \frac{4750757}{4}k^7 - \frac{8575621}{8}k^6 + 720801k^5 \\
& - \frac{2801415}{8}k^4 + \frac{464797}{4}k^3 - \frac{47153}{2}k^2 + 2207k
\end{align*}
\begin{align*}
\pi_{T_2}^{(Q_3)}(k)&= \frac{55}{2}k^{15} - 550k^{14} + \frac{10147}{2}k^{13} - 28609k^{12} + \frac{881345}{8}k^{11} - \frac{2451731}{8}k^{10} \\
& + \frac{5080851}{8}k^9 - \frac{7968631}{8}k^8 + \frac{9504849}{8}k^7 - \frac{8580805}{8}k^6 + \frac{5771941}{8}k^5 \\
& - \frac{2805421}{8}k^4 + \frac{465739}{4}k^3 - \frac{47283}{2}k^2 + 2215k
\end{align*}

We present the following pair of trees $T_1$ and $T_2$ for which $\pi^{(Q_3)}_{T_1}(k) = \pi^{(Q_3)}_{T_2}(k)$, but $M_3(T_1)\ne M_3(T_2)$
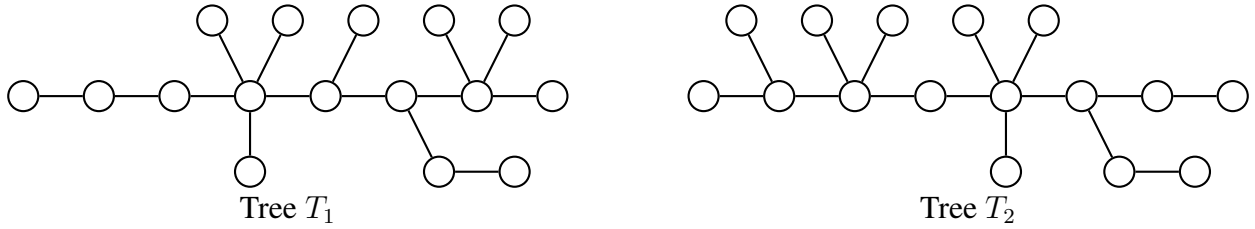
\begin{figure}[ht]
    \centering
    \begin{tikzpicture}
        \node[vertex] (S1) at (0,0) {};
        \node[vertex] (S2) at (1,0) {};
        \node[vertex] (S3) at (2,0) {};
        \node[vertex] (S4) at (3,0) {};
        \node[vertex] (S4L) at (2.5,1) {};
        \node[vertex] (S4D) at (3,-1) {};
        \node[vertex] (S4R) at (3.5,1) {};
        \node[vertex] (S5) at (4,0) {};
        \node[vertex] (S5U) at (4.5,1) {};
        \node[vertex] (S6) at (5,0) {};
        \node[vertex] (S6D1) at (5.5,-1) {};
        \node[vertex] (S6D2) at (6.5,-1) {};
        \node[vertex] (S7) at (6,0) {};
        \node[vertex] (S7U) at (6.5,1) {};
        \node[vertex] (S7D) at (5.5,1) {};
        \node[vertex] (S8) at (7,0) {};
        \draw[edge] (S1) -- (S2) -- (S3) -- (S4) -- (S5) -- (S6) -- (S7) -- (S8);
        \draw[edge] (S4L) -- (S4) -- (S4R);
        \draw[edge] (S4D) -- (S4);
        \draw[edge] (S5U) -- (S5);
        \draw[edge] (S6) -- (S6D1) -- (S6D2);
        \draw[edge] (S7D) -- (S7) -- (S7U);
        \node at (3.5, -1.5) {Tree $T_1$};

        \node[vertex] (T1) at (9,0) {};
        \node[vertex] (T2) at (10,0) {};
        \node[vertex] (T2U) at (9.5,1) {};
        \node[vertex] (T3) at (11,0) {};
        \node[vertex] (T3L) at (10.5,1) {};
        \node[vertex] (T3R) at (11.5,1) {};
        \node[vertex] (T4) at (12,0) {};
        \node[vertex] (T5) at (13,0) {};
        \node[vertex] (T5L) at (12.5,1) {};
        \node[vertex] (T5R) at (13.5,1) {};
        \node[vertex] (T5D) at (13,-1) {};
        \node[vertex] (T6) at (14,0) {};
        \node[vertex] (T6D1) at (14.5,-1) {};
        \node[vertex] (T6D2) at (15.5,-1) {};
        \node[vertex] (T7) at (15,0) {};
        \node[vertex] (T8) at (16,0) {};
        \draw[edge] (T1) -- (T2) -- (T3) -- (T4) -- (T5) -- (T6) -- (T7) -- (T8);
        \draw[edge] (T6) -- (T6D1) -- (T6D2);
        \draw[edge] (T2) -- (T2U);
        \draw[edge] (T3L) -- (T3) -- (T3R);
        \draw[edge] (T5L) -- (T5) -- (T5R);
        \draw[edge] (T5) -- (T5D);
        \node at (12.5, -1.5) {Tree $T_2$};
    \end{tikzpicture}
\end{figure}
\begin{align*}
\pi_{T_1}^{(Q_4)}(k) = \pi_{T_2}^{(Q_4)}(k) 
&= 70k^{19}-1680k^{18}+\frac{75999}{4}k^{17}-\frac{1076337}{8}k^{16}\\
&+668493k^{15}-\frac{9899537}{4}k^{14}+\frac{28300291}{4}k^{13}-\frac{63874601}{4}k^{12}\\
&+28845446k^{11}-42011872k^{10}+\frac{197796569}{4}k^9-\frac{375380081}{8}k^8+35626802k^7\\
&-\frac{85398283}{4}k^6+\frac{39499621}{4}k^5-\frac{6803041}{2}k^4+821653k^3-124194k^2+8840k
\end{align*}

However, in $T_1$, there are only 2 triplets of vertices with 2 internal edges and 6 external edges, and in $T_2$ there are 3 such triplets.

Next, we describe connections to the chromatic symmetric function (CSF) of a tree. We begin by recalling the definition of the CSF. If a graph $G$ has the vertex set $\{v_1, \ldots, v_n\}$, the CSF of $G$ is defined by a power series in infinitely many variables:
$$
X_G = X_G(x_1, x_2, ...) = \sum_{c} x_{c(v_1)} x_{c(v_2)} \cdots  x_{c(v_n)}
$$
where the sum ranges over all proper colorings $c\colon V(G)\to \mathbb{N}$. Stanley asked whether two trees sharing the same CSF must be isomorphic in \cite{Sta95}. Recently, \cite{APMWZ24} proved that if two trees share the same CSF, then they share the same generalized degree sequence, confirming Crew's conjecture~\cite{Cre20}. Since $\pi_{T}^{(Q_d)}$ is determined by the generalized degree sequence, we have the following consequence.

\begin{cor}\label{cor:CSF:hypercube}
If $T_1$ and $T_2$ are two trees with the same CSF, then $\pi_{T_1}^{(Q_d)} = \pi_{T_2}^{(Q_d)}$ for all $d\geq 0$. 
\end{cor}

To address Stanley's question, it would be helpful to extend Corollary~\ref{cor:CSF:hypercube} to conclude equality $\pi_{T_1}^{(H)} = \pi_{T_2}^{(H)}$ for other graphs $H$. 

In \cite{AKL25}, it is shown that $\mathcal{C}_3(T)$ contains a special subgraph which is not present in a $3$-coloring graph of any other tree. In particular, for a tree $T$ on $n$ vertices, there exists a graph $H(T)$ on $\binom{n}{2}+2$ vertices and $n^2-2n+2$ edges such that $\pi^{(H(T))}_T(3)>0$. However, for any other tree $G$ on at most $n$ vertices, $\pi^{(H(T))}_G(3)=0$. Thus, the infinite family of polynomials $\{\pi^{(H(T))}\}_{T}$ is a complete tree invariant. Recovering $\pi^{(H(T))}$ from $X_T$ would answer Stanley's question in the affirmative.

To differentiate all trees on $n$ vertices or fewer, the number of polynomials to check using this method grows exponentially in $n$. This is because there is a bijection $T \leftrightarrow H(T)$ between trees and their associated polynomials. By reducing the number of polynomials to check, it may become more tractable to extract them from the CSF.

\section{Counterexample to Conjecture 5.2 of Asgarli et al.}\label{sec:counterexample}
In this section we demonstrate the computational advantages of the shadow graph method of counting induced subgraphs over previous methods, especially for smaller graphs. Calculation of the chromatic polynomial of a graph $G$ can be done relatively quickly if the chromatic polynomials of many subgraphs of $G$ are known. The number of chromatic polynomials necessary to construct $\pi^{(P_2)}_G$ grows linearly with the order of $G$, making calculation of $pi^{(P_2)}_G$ using shadow graphs relatively inexpensive compared to previous methods. This computational advantage has allowed us to perform a deeper direct search of chromatic pairs polynomials and find a counterexample to a conjecture posed by Asgarli et al.

\subsection{Statement of the Conjecture}

Asgarli et al. state the following conjecture in their work:

\begin{quote}
\textbf{Conjecture 5.2 (Asgarli et al.\ \cite{AKLR25})}: If two graphs $G_1$ and $G_2$ satisfy $\pi_{G_1}^{(P_2)}(k) = \pi_{G_2}^{(P_2)}(k)$, then $\pi_{G_1}(k) = \pi_{G_2}(k)$.
\end{quote}

This conjecture suggests that if two graphs have identical chromatic pairs polynomials, they must also have identical chromatic polynomials. Indeed no counterexample was found on graphs $G$ of order at most 7. We present a counterexample to this conjecture using two graphs of order 8 with equal $\pi^{(P_2)}(k)$ but differing $\pi(k)$.

\subsection{The Counterexample Graphs}

We present the two graphs $ G_1 $ and $ G_2 $ that form our counterexample.

\begin{figure}[ht]
    \centering
    \begin{tikzpicture}
        \node[vertex] (A1) at (0,0) {};
        \node[vertex] (B1) at (1,0) {};
        \node[vertex] (C1) at (1,1) {};
        \node[vertex] (D1) at (0,1) {};
        \draw[edge] (A1) -- (B1) -- (C1) -- (D1) -- (A1) -- (C1);
        \node[vertex] (A2) at (2,0) {};
        \node[vertex] (B2) at (3,0) {};
        \node[vertex] (C2) at (3,1) {};
        \node[vertex] (D2) at (2,1) {};
        \draw[edge] (A2) -- (B2) -- (C2) -- (D2) -- (A2) -- (C2);
        \node[vertex] (A) at (6,0) {};
        \node[vertex] (B) at (7,0) {};
        \node[vertex] (C) at (7,1) {};
        \node[vertex] (D) at (6,1) {};
        \node[vertex] (E) at (8,0.5) {};
        \node[vertex] (F) at (5, 0.5) {};
        \node[vertex] (G) at (8,0) {};
        \node[vertex] (H) at (8,1) {};
        \draw[edge] (A) -- (B) -- (C) -- (D) -- (A) -- (C);
        \draw[edge] (B) -- (D);
        \draw[edge] (B) -- (E);
        \draw[edge] (C) -- (E);
        \draw[edge] (G) -- (B);
        \draw[edge] (H) -- (C);
        \node at (1.5, -0.5) {Graph $G_1$};
        \node at (6.5, -0.5) {Graph $G_2$};
        
    \end{tikzpicture}

    \caption{Graphs $ G_1 $ and $ G_2 $ used as a counterexample to Conjecture 5.2.}
    \label{fig:counterexample-graphs}
\end{figure}
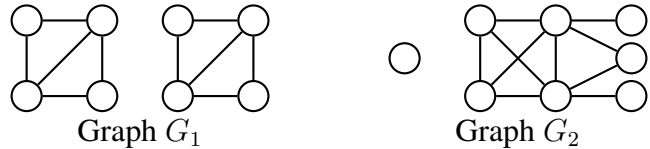
The chromatic pairs polynomials of both graphs are as follows:
$$
\pi_{G_1}^{(P_2)}(k) = \pi_{G_2}^{(P_2)}(k) = 4k^9 - 54k^8 + 306k^7 - 942k^6 + 1698k^5 - 1788k^4 + 1016k^3 - 240k^2.
$$
The chromatic polynomials of the graphs are:
$$
\pi_{G_1}(k) = k^8 - 10k^7 + 41k^6 - 88k^5 + 104k^4 - 64k^3 + 16k^2.
$$
$$
\pi_{G_2}(k) = k^8 - 10k^7 + 40k^6 - 82k^5 + 91k^4 - 52k^3 + 12k^2.
$$
As seen, $\pi_{G_1}(k) \neq \pi_{G_2}(k)$.

Another counterexample involving connected graphs is the following:
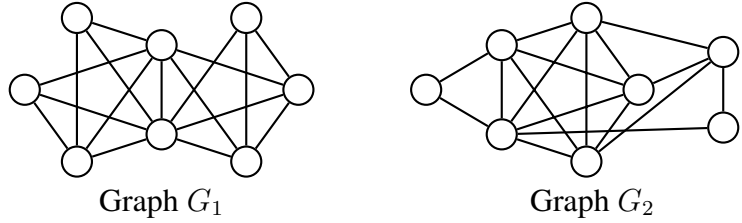
\begin{figure}[ht]
    \centering
    \begin{tikzpicture}
        \node[vertex] (S1) at (0, 0.59) {}; 
        \node[vertex] (S2) at (0, -0.59) {};
        \node[vertex] (A1) at (-1.12, -0.95) {};
        \node[vertex] (A2) at (-1.81, 0.00) {};
        \node[vertex] (A3) at (-1.12, 0.95) {};
        \node[vertex] (B1) at (1.12, -0.95) {};
        \node[vertex] (B2) at (1.81, 0.00) {};
        \node[vertex] (B3) at (1.12, 0.95) {};
        \draw[edge] (S1) -- (S2);
        \draw[edge] (S1) -- (A1);
        \draw[edge] (S1) -- (A2);
        \draw[edge] (S1) -- (A3);
        \draw[edge] (S2) -- (A1);
        \draw[edge] (S2) -- (A2);
        \draw[edge] (S2) -- (A3);
        \draw[edge] (A1) -- (A3);
        \draw[edge] (A2) -- (A1);
        \draw[edge] (S1) -- (B2);
        \draw[edge] (S1) -- (B1);
        \draw[edge] (S2) -- (B1);
        \draw[edge] (S2) -- (B2);
        \draw[edge] (S2) -- (B3);
        \draw[edge] (B1) -- (B2);
        \draw[edge] (B1) -- (B3);
        \draw[edge] (B2) -- (B3);
        \node[vertex] (L) at (3.5, 0) {};
        \node[vertex] (Z1) at (4.5, 0.59) {}; 
        \node[vertex] (Z2) at (4.5, -0.59) {};
        \node[vertex] (C1) at (5.62, -0.95) {};
        \node[vertex] (C2) at (6.31, 0.00) {};
        \node[vertex] (C3) at (5.62, 0.95) {};
        \node[vertex] (R1) at (7.43,0.5) {};
        \node[vertex] (R2) at (7.43,-0.5) {};
        \draw[edge] (Z1) -- (L) -- (Z2) -- (Z1) -- (C3) -- (C2) -- (C1) -- (Z2) -- (C3) -- (C1) -- (Z1) -- (C2) -- (Z2);
        \draw[edge] (C1) -- (R1) -- (C2);
        \draw[edge] (Z2) -- (R2) -- (R1) -- (C3);
        \node at (0, -1.5) {Graph $G_1$};
        \node at (5.7, -1.5) {Graph $G_2$};
    \end{tikzpicture}
\end{figure}
$$
\pi_{G_1}^{(P_2)}(k) = \pi_{G_2}^{(P_2)}(k) = 8k^9-178k^8+\frac{1711}{2}k^7-\frac{9265}{2}k^6+\frac{30841}{2}k^5-\frac{64417}{2}k^4+41032k^3-28950k^2+8568k.
$$
$$
\pi_{G_1}(k) =  k^8-17k^7+122k^6-479k^5+1109k^4-1508k^3+1108k^2-336k.
$$
$$
\pi_{G_2}(k) = k^8-17k^7+122k^6-478k^5+1101k^4-1485k^3+1080k^2-324k.
$$
\section{Conclusion}

In this paper, we introduced a novel method based on \emph{shadow graphs} to explicitly compute the polynomials $\pi_G^{(H)}(k)$, counting induced copies of a graph $H$ within the $k$-coloring graph $\mathcal{C}_k(G)$ of a base graph $G$. This method provides a constructive proof for the polynomial nature of $\pi_G^{(H)}(k)$, complementing existing proofs by Asgarli et al. \cite{AKLR25} and Hogan et al. \cite{HSTT24}. We demonstrated the practicality of this construction conceptually by determining the strength of all hypercube counting polynomials as tree invariants, as well as computationally by implementing the chromatic pairs polynomial to find a counterexample to a conjecture by Asgarli et al. \cite{AKLR25}. Looking ahead, there are several avenues for future research. The connection between $\pi_T^{(Q_d)}(k)$ and the generalized degree sequence, combined with the recent result that the CSF determines the generalized degree sequence, naturally leads to Corollary~\ref{cor:CSF:hypercube}. Determining which other $\pi_T^{(H)}(k)$ polynomials can be readily determined from the CSF would likely further illuminate the potential of the CSF as a complete tree invariant. 

Overall, shadow graphs provide a powerful and explicit tool for exploring the structure of coloring graphs and their connection to fundamental graph invariants. The results presented here open up new computational possibilities and theoretical questions about graph colorings, reconfiguration problems, and graph invariants.
\section{Statements and Declarations}
\textbf{Conflict of Interest Statement.} The author declares that there is no conflict of interest arising from this work.

\textbf{Data Availability Statement.} No data sets were used in the current manuscript.
\bibliographystyle{alpha}
\bibliography{main}

\end{document}